\documentclass{article} 

\usepackage{amsmath,amsthm,amssymb,amsfonts}

\usepackage{geometry} 
\newtheorem{Theorem}{Theorem}[section]
\theoremstyle{plain}

\newtheorem{Lemma}[Theorem]{Lemma}

\newtheorem{Proposition}[Theorem]{Proposition}

\numberwithin{equation}{section}

\title{Fluctuations of the rightmost particle in the catalytic branching Brownian motion}
\author{Sergey Bocharov\footnote{S.Bocharov:  
Department of Mathematics, Zhejiang University, Zheda Road, Hangzhou 310027, China, 
e-mail: bocharov@zju.edu.cn. The author is supported by NSFC grant (No.11731012)}}
\begin{document}
\maketitle
\begin{abstract}
In this article we establish the magnitude of fluctuations of the extreme particle in the model of binary branching Brownian motion with a single catalytic point at the origin.
\end{abstract}

\section{Introduction}

\subsection{Description of the model}

Catalytic branching Brownian motion (BBM) is a spatial population model in which individuals, referred to as `particles', move in space according to the law of a standard Brownian motion and reproduce themselves at a spatially-inhomogeneous rate $\beta \delta_0(\cdot)$, where $\beta$ is some positive constant and $\delta_0$ is the Dirac delta measure.

In our model we start with one particle at some initial location $x_0 \in \mathbb{R}$ at time $0$. The position of this particle at time $t \geq 0$ up until the time when it dies is given by a standard Brownian motion $(X_t)_{t \geq 0}$. If $(L_t)_{t \geq 0}$ is the local time at the origin of the process $(X_t)_{t \geq 0}$ then at a random time $T$ satisfying 
\[
\mathbb{P}^{x_0} \big(T > t \big\vert (X_s)_{s \geq 0} \big) = \mathrm{e}^{- \beta L_t},
\]
the initial particle dies and is replaced with two new particles, which independently of each other and of the past stochastically continue the behaviour of their parent starting at time $T$ and position $X_T = 0$. That is, they move in space as Brownian motions, die after random times giving birth to two new particles each and so on.

\subsection{Notation and earlier results}


We let $N_t$ denote the set of all the particles alive at time $t$. For each of the particles $u \in N_t$ we let $X^u_t$ be its spatial position at the given time $t$ and $(X^u_s)_{0 \leq s \leq t}$ its path up to time $t$.

We let $(\mathcal{F}_t)_{t \geq 0}$ be the natural filtration of the branching process. We also denote the law of the branching process initiated from $x_0$ by $P^{x_0}$ with the corresponding expectation $E^{x_0}$. When $x_0 = 0$ we would write $P$ and $E$ rather than $P^0$ and $E^0$.

We define 
\[
R_t := \sup_{u \in N_t} X^u_t
\]
to be the position of the rightmost particle at time $t$. We are interested in the asymptotic behaviour of $R_t$ as $t \to \infty$. It was already shown in \cite{BH14} that
\begin{equation}
\label{rightmost_as}
\frac{R_t}{t} \to \frac{\beta}{2} \qquad P^{x_0} \text{-a.s.}
\end{equation}
and then later in \cite{BH16} that 
\begin{equation}
\label{rightmost_law}
R_t - \frac{\beta}{2}t \Rightarrow W,
\end{equation}
where
\[
P^{x_0} (W \leq x) = E^{x_0} \Big[ \exp \big\{ - \mathrm{e}^{- \beta x} M_\infty \big\} \Big] \quad \text{ , } x \in \mathbb{R}.
\]
Here $M_\infty$ is the almost sure limit of the ``additive" martingale
\begin{equation}
\label{add_mart}
M_t = \mathrm{e}^{- \frac{\beta^2}{2}t} \sum_{u \in N_t} \mathrm{e}^{- \beta |X^u_t|}
\end{equation}
and it is known from \cite{BH14} that $M_\infty > 0$ $P^{x_0}$-almost surely.

Let us mention that versions of \eqref{rightmost_as} and \eqref{rightmost_law} for BBM with branching rate given by a continuous function decaying sufficiently fast at infinity were proved in \cite{E84} and \cite{LS88} respectively. Versions of \eqref{rightmost_as} and \eqref{rightmost_law} for BBM with branching rate given by measures decaying sufficiently fast at infinity were proved in \cite{S18} and \cite{NS19} respectively. Versions of \eqref{rightmost_as} and \eqref{rightmost_law} for discrete-time catalytic branching random walks on $\mathbb{Z}$ were proved in \cite{CH14}. There are also numerous versions of \eqref{rightmost_as} and \eqref{rightmost_law} for models with ``spatially-homogeneous" branching.
\subsection{The main result}

As a trivial corollary to \eqref{rightmost_law} one gets that for any function $f(\cdot)$ with $\lim_{t \to \infty} f(t) = \infty$, 
\[
\frac{R_t - \frac{\beta}{2}t}{f(t)} \to 0 \qquad \text{ in } P^{x_0} \text{ probability}.
\]
This is a rather weak statement and one should really be interested in the almost sure asymptotic behaviour of $\frac{R_t - \frac{\beta}{2}t}{f(t)}$ for different functions $f(\cdot)$, which is summarised in the following theorem.
\begin{Theorem}
\label{main}
Take $x_0 \in \mathbb{R}$ and let $R_t$ be as above then 
\begin{equation}
\label{main_limsup}
\limsup_{t \to \infty} \frac{R_t - \frac{\beta}{2}t}{\log t} = \frac{1}{\beta} \qquad P^{x_0} \text{-a.s.}
\end{equation}
and
\begin{equation}
\label{main_liminf}
\liminf_{t \to \infty} \frac{R_t - \frac{\beta}{2}t}{\log \log t} = - \frac{1}{\beta} \qquad P^{x_0} \text{-a.s.}
\end{equation}
\end{Theorem}
In particular, we also get the lighter version of \eqref{main_liminf}:
\begin{equation}
\label{liminf_trivial}
\liminf_{t \to \infty} \frac{R_t - \frac{\beta}{2}t}{\log t} = 0 \qquad P^{x_0} \text{-a.s.}
\end{equation}
Let us mention that versions of Theorem \ref{main} have been known for models with ``spatially-homogeneous" branching (see, for example, \cite{HS09}, \cite{R13} or \cite{RSZ19}) but, to the best of our knowledge, there has not so far been any corresponding result for a model with ``spatially-inhomogeneous" branching.
\subsection{Outline of the paper}

The rest of the paper is organised as follows. In Section 2 we present some preliminary simple results, many of which are known from earlier work. Section 3 contains the proof of Theorem \ref{main}, which is divided in four parts: in Subsection 3.1 it is proved that
\[
\limsup_{t \to \infty} \frac{R_t - \frac{\beta}{2}t}{\log t} \leq \frac{1}{\beta} \qquad P^{x_0} \text{-a.s.},
\]
in Subsection 3.2 it is proved that
\[
\limsup_{t \to \infty} \frac{R_t - \frac{\beta}{2}t}{\log t} \geq \frac{1}{\beta} \qquad P^{x_0} \text{-a.s.},
\]
in Subsection 3.3 it is proved that
\[
\liminf_{t \to \infty} \frac{R_t - \frac{\beta}{2}t}{\log \log t} \geq - \frac{1}{\beta} \qquad P^{x_0} \text{-a.s.}
\]
and in Subsection 3.4 it is proved that
\[
\liminf_{t \to \infty} \frac{R_t - \frac{\beta}{2}t}{\log \log t} \leq - \frac{1}{\beta} \qquad P^{x_0} \text{-a.s.}
\]

\section{Preliminaries}

It is a common practice to extend the original probability space of the branching system by adding the spine process to it. The spine is an infinite line of descent which begins with the initial particle and whenever the particle presently in the spine dies one of its two children is chosen with probability $\frac{1}{2}$ to continue the spine independently of all the previous history. 

If we then let $\tilde{P}$ denote the extension of the original probability measure $P$ to this bigger probability space and if at every $t \geq 0$ we let $\xi_t$ denote the spatial position of the spine particle at time $t$ then one can see that the process $(\xi_t)_{t \geq 0}$ is a Brownian motion under $\tilde{P}$. If we let $(\tilde{L}_t)_{t \geq 0}$ denote the local time of $(\xi_t)_{t \geq 0}$ at the origin then the following result is known to hold (see e.g. \cite{BH14}, Theorem 5).
\begin{Lemma}[Many-to-one lemma]
\label{many_to_one}
Take $x_0 \in \mathbb{R}$ and $t \geq 0$. Then for $f : \mathcal{C}\big([0, t]\big) \to [0, \infty)$ a Borel measurable function we have
\begin{equation}
\label{eq_many_to_one}
E^{x_0} \Big[ \sum_{u \in N_t} f\big( (X^u_s )_{0 \leq s \leq t} \big) \Big] = \tilde{E}^{x_0} \Big[ f\big( (\xi_s )_{0 \leq s \leq t} \big) \mathrm{e}^{\beta \tilde{L}_t} \Big].
\end{equation}
\end{Lemma}
A typical application of Lemma \ref{many_to_one} is to calculate the expected number of particles at time $t$ whose paths satisfy a certain condition. For example, for $x \in \mathbb{R}$ and $t \geq 0$ let us define 
\[
N_t^x := \big\{ u \in N_t : X^u_t \geq x \big\}
\]
to be the set of particles at time $t$ which lie above level $x$. Then 
\[
E^{x_0} |N_t^x| = E^{x_0} \Big[ \sum_{u \in N_t} \mathbf{1}_{\{X^u_t \geq x\}} \Big] = \tilde{E}^{x_0} \Big[ \mathrm{e}^{\beta \tilde{L}_t} \mathbf{1}_{\{ \xi_t \geq x \}} \Big]
\]
and below we give the exact expression for this expectation.
\begin{Proposition}
\label{expectation}
For any $x_0 \in \mathbb{R}$, $t \geq 0$ and $x \geq 0$
\begin{align}
\label{eq_expectation}
E^{x_0} |N_t^x| = &\tilde{E}^{x_0} \Big[ \mathrm{e}^{\beta \tilde{L}_t} \mathbf{1}_{\{ \xi_t \geq x \}} \Big]\nonumber\\
= &\mathrm{e}^{- \beta |x_0| - \beta x + \frac{\beta^2}{2}t} \Phi \Big( \frac{\beta t - |x_0| - x}{\sqrt{t}} \Big) 
+ \Big[ \Phi \Big( \frac{x + x_0}{\sqrt{t}} \Big) - \Phi \Big( \frac{x - x_0}{\sqrt{t}} \Big) \Big] \mathbf{1}_{\{x_0 \geq 0\}},
\end{align}
where $\Phi(z) = \int_{- \infty}^z \frac{1}{\sqrt{2 \pi}} \mathrm{e}^{- \frac{y^2}{2}} \mathrm{d}y$, $z \in \mathbb{R}$ is the cumulative distribution function of a $\mathcal{N}(0, 1)$ random variable.
\end{Proposition}
In particular, it is always true that
\begin{equation}
\label{eq_expectation00}
E^{x_0} |N_t^x| \leq \mathrm{e}^{- \beta |x_0| - \beta x + \frac{\beta^2}{2}t} + 1 - \Phi \Big( \frac{x - x_0}{\sqrt{t}} \Big) = \mathrm{e}^{- \beta |x_0| - \beta x + \frac{\beta^2}{2}t} + \Phi \Big( - \frac{x - x_0}{\sqrt{t}} \Big)
\end{equation}
and in the special case when $x_0 = 0$ it is true that
\begin{equation}
\label{eq_expectation0}
E |N_t^x|= \mathrm{e}^{- \beta x + \frac{\beta^2}{2}t} \Phi \Big( \frac{\beta t - x}{\sqrt{t}} \Big) \leq 
\mathrm{e}^{- \beta x + \frac{\beta^2}{2}t}.
\end{equation}
\begin{proof}[Proof of Proposition \ref{expectation}]
Proposition \ref{expectation} was essentially proved in \cite{B20} (see Proposition 4). It was already shown there in a slightly more general setting that 
\[
E^{x_0} |N_t^x| = \int_x^\infty \frac{1}{\sqrt{2 \pi t}} \mathrm{e}^{- \frac{1}{2}(y - x_0)^2} \mathrm{d}y + 
\mathrm{e}^{- \beta |x_0| + \frac{\beta^2}{2}t} \int_x^{\infty} \beta \mathrm{e}^{- \beta y} \Phi \Big( \frac{\beta t - |x_0| - y}{\sqrt{t}} \Big) \mathrm{d}y
\]
A simple application of the integration-by-parts formula to the last integral gives us 
\begin{align*}
E^{x_0} |N_t^x| = &\int_x^\infty \frac{1}{\sqrt{2 \pi t}} \mathrm{e}^{- \frac{1}{2t}(y - x_0)^2} \mathrm{d}y\\ 
&+ \mathrm{e}^{- \beta |x_0| + \frac{\beta^2}{2}t} \bigg( \Big[ - \mathrm{e}^{-\beta y} \Phi \Big( \frac{\beta t - |x_0| - y}{\sqrt{t}} \Big) \Big]_x^\infty - \int_x^\infty \mathrm{e}^{- \beta y} \frac{1}{\sqrt{2 \pi t}} \mathrm{e}^{- \frac{1}{2t} (\beta t - |x_0| - y)^2} \mathrm{d}y\bigg)\\
= &\int_x^\infty \frac{1}{\sqrt{2 \pi t}} \mathrm{e}^{- \frac{1}{2t}(y - x_0)^2} \mathrm{d}y\\
&+ \mathrm{e}^{- \beta |x_0| - \beta x + \frac{\beta^2}{2}t} \Phi \Big( \frac{\beta t - |x_0| - x}{\sqrt{t}} \Big) - \int_x^\infty \frac{1}{\sqrt{2 \pi t}} \mathrm{e}^{- \frac{1}{2t}(y + |x_0|)^2} \mathrm{d}y,
\end{align*}
which yields the required result.
\end{proof}
Also, with the careful use of symmetry we may deduce from \eqref{eq_expectation} that
\begin{equation}
\label{exp_population}
E^{x_0} |N_t| = \tilde{E}^{x_0} \big[ \mathrm{e}^{\beta \tilde{L}_t} \big] = 2 \mathrm{e}^{-\beta |x_0| + \frac{\beta^2}{2}t} \Phi \Big( \beta \sqrt{t} - \frac{|x_0|}{\sqrt{t}}\Big) + \Phi \Big(\frac{|x_0|}{\sqrt{t}} \Big) - \Phi \Big( - \frac{|x_0|}{\sqrt{t}} \Big).
\end{equation}
In particular, it is always true that 
\begin{equation}
\label{ineq_exp}
E^{x_0} |N_t| \leq 1 + 2 \mathrm{e}^{- \beta |x_0| + \frac{\beta^2}{2}t}
\end{equation}
and in the special case when $x_0 = 0$ it is true that
\begin{equation}
\label{ineq_exp1}
E |N_t| = 2 \mathrm{e}^{\frac{\beta^2}{2}t} \Phi(\beta \sqrt{t}) \leq 2 \mathrm{e}^{\frac{\beta^2}{2}t}.
\end{equation}
The next general result, among other things, allows us to calculate the second moment of $|N_t^x|$ (for the proof see \cite{BH16}, Lemma 2.3  or \cite{S08}, Lemma 3.3). 
\begin{Lemma}[Many-to-two lemma]
\label{many_to_two}
Take $x_0 \in \mathbb{R}$ and $t \geq 0$. Then for $f : \mathbb{R} \to [0, \infty)$ and 
$g : \mathbb{R} \to [0, \infty)$ Borel measurable functions we have
\begin{equation}
\label{eq_many_to_two}
E^{x_0} \Big[ \Big( \sum_{u \in N_t} f(X^u_t) \Big) \Big( \sum_{u \in N_t} g(X^u_t) \Big) \Big] = S^{x_0}_{fg}(t) + 
2 \tilde{E}^{x_0} \Big[ \int_0^t S^0_f(t - \tau) S^0_g(t - \tau) \mathrm{d} \big( \mathrm{e}^{\beta \tilde{L}_\tau} \big) \Big],
\end{equation}
where 
\[
S^{x_0}_f(t) = E^{x_0} \Big[ \sum_{u \in N_t} f (X^u_t ) \Big],
\]
which can be computed using Lemma \ref{many_to_one}.
\end{Lemma}
The following inequality is useful when we need to estimate the second term in \eqref{eq_many_to_two}.
\begin{Proposition}
\label{expectation2}
For any $x_0 \in \mathbb{R}$ and $t \geq 0$
\begin{equation}
\label{eq_expectation2}
\tilde{E}^{x_0} \Big[ \int_0^t \mathrm{e}^{- \beta^2 \tau} \mathrm{d} \big( \mathrm{e}^{\beta \tilde{L}_\tau} \big) \Big] \leq 4 \mathrm{e}^{- \beta |x_0|}.
\end{equation}
\end{Proposition}
\begin{proof}[Proof of Proposition \ref{expectation2}]
Using integration by parts, Fubini's Theorem and inequality \eqref{ineq_exp} we get 
\begin{align*}
\tilde{E}^{x_0} \Big[ \int_0^t \mathrm{e}^{- \beta^2 \tau} \mathrm{d} \big( \mathrm{e}^{\beta \tilde{L}_\tau} \big) \Big] = &\tilde{E}^{x_0} \Big[ \big[ \mathrm{e}^{- \beta^2 \tau} \mathrm{e}^{\beta \tilde{L}_\tau} \big]_0^t + \int_0^t \mathrm{e}^{\beta \tilde{L}_\tau} \ \beta^2 \mathrm{e}^{- \beta^2 \tau} \mathrm{d}\tau \Big]\\
= &\mathrm{e}^{- \beta^2 t} \tilde{E}^{x_0} \big[ \mathrm{e}^{\beta \tilde{L}_t} \big] - 1 + \int_0^t \tilde{E}^{x_0} \big[ \mathrm{e}^{\beta \tilde{L}_\tau} \big] \beta^2 \mathrm{e}^{- \beta^2 \tau} \mathrm{d}\tau\\
\leq &\mathrm{e}^{- \beta^2 t} \big( 1 + 2 \mathrm{e}^{- \beta |x_0| + \frac{\beta^2}{2}t} \big) - 1 + \int_0^t \big( 1 + 2 \mathrm{e}^{- \beta |x_0| + \frac{\beta^2}{2}\tau} \big) \beta^2 \mathrm{e}^{- \beta^2 \tau} \mathrm{d}\tau\\
= &4\mathrm{e}^{- \beta |x_0|} - 2\mathrm{e}^{- \beta |x_0| - \frac{\beta^2}{2}t},
\end{align*}
which proves the result.
\end{proof}
As a corollary to Lemma \ref{many_to_two} and Proposition \ref{expectation2} we get the following useful inequlities.
\begin{Proposition}
\label{second_moment_prop}
For any $x_0 \in \mathbb{R}$, $t \geq 0$ and $x \geq 0$
\begin{equation}
\label{second_moment_ineq}
E^{x_0} \big[ |N_t^x|^2 \big] \leq E^{x_0} |N_t^x| + 8 \mathrm{e}^{- \beta |x_0| - 2 \beta x + \beta^2 t}.
\end{equation}
\end{Proposition}
\begin{proof}
Taking $f(\cdot) = g(\cdot) = \mathbf{1}_{[x, \infty)}(\cdot)$ in identity \eqref{eq_many_to_two} and then applying inequalitites \eqref{eq_expectation0} and \eqref{eq_expectation2} we get
\begin{align*}
E^{x_0} \big[ |N_t^x|^2 \big] = &E^{x_0} |N_t^x| + 2 \tilde{E}^{x_0} \Big[ \int_0^t \big( E |N_{t - \tau}^x|\big)^2
\mathrm{d} \big( \mathrm{e}^{\beta \tilde{L}_\tau} \big) \Big]\\
\leq &E^{x_0} |N_t^x| + 2 \tilde{E}^{x_0} \Big[ \int_0^t \big( \mathrm{e}^{- \beta x + \frac{\beta^2}{2}(t - \tau)} \big)^2
\mathrm{d} \big( \mathrm{e}^{\beta \tilde{L}_\tau} \big) \Big]\\
\leq &E^{x_0} |N_t^x| + 2 \mathrm{e}^{- 2\big(\beta x + \frac{\beta^2}{2}t\big)}  \tilde{E}^{x_0} \Big[ \int_0^t \mathrm{e}^{ - \beta^2 \tau} \mathrm{d} \big( \mathrm{e}^{\beta \tilde{L}_\tau} \big) \Big]\\
\leq &E^{x_0} |N_t^x| + 8 \mathrm{e}^{- \beta |x_0| - 2  \beta x + \beta^2 t}.
\end{align*}
\end{proof}
\begin{Proposition}
\label{second_moment_prop2}
For any $x_0 \in \mathbb{R}$, $t \geq 0$ and $x \geq 0$
\begin{equation}
\label{second_moment_ineq2}
E^{x_0} \big[ |N_t^x| |N_t| \big] \leq E^{x_0} |N_t^x| + 16 \mathrm{e}^{- \beta |x_0| - \beta x + \beta^2 t }.
\end{equation}
\end{Proposition}
\begin{proof}
Taking $f(\cdot) = \mathbf{1}_{[x, \infty)}(\cdot)$ and $ g(\cdot) = 1$ in identity \eqref{eq_many_to_two} and then applying inequalitites \eqref{eq_expectation0}, \eqref{ineq_exp1} and \eqref{eq_expectation2} we get
\begin{align*}
E^{x_0} \big[ |N_t^x| |N_t| \big] = &E^{x_0} |N_t^x| + 2 \tilde{E}^{x_0} \Big[ \int_0^t E |N_{t - \tau}^x| E|N_{t -\tau}|
\mathrm{d} \big( \mathrm{e}^{\beta \tilde{L}_\tau} \big) \Big]\\
\leq &E^{x_0} |N_t^x| + 2 \tilde{E}^{x_0} \Big[ \int_0^t 2 \mathrm{e}^{- \beta x + \beta^2 (t - \tau)}
\mathrm{d} \big( \mathrm{e}^{\beta \tilde{L}_\tau} \big) \Big]\\
\leq &E^{x_0} |N_t^x| + 16 \mathrm{e}^{- \beta |x_0| -  \beta x + \beta^2 t}.
\end{align*}
\end{proof}
The final result of this section is an upper bound on $E^{x_0} \big[ |N_s^x| |N_t^y| \big]$, which, unfortunately, cannot take a simple form. 
\begin{Proposition}
\label{second_moment_prop3}
For any $x_0 \in \mathbb{R}$, $0 < s < t$ and $0 \leq x < y$
\begin{equation}
\label{second_moment_ineq3}
E^{x_0} \big[ |N_s^x| |N_t^y| \big] \leq 24 \mathrm{e}^{- \beta |x_0| - \beta x - \beta y + \frac{\beta^2}{2}s + \frac{\beta^2}{2}t} + \sum_{i = 1}^4 e_i(x_0, s, t, x, y),
\end{equation}
where
\begin{align*}
e_1(x_0, s, t, x, y) = &\bigg( \mathrm{e}^{- \beta y + \frac{\beta^2}{2}(t-s)} + \Phi \Big( - \frac{y-x}{\sqrt{t-s}} \Big) \bigg) E^{x_0} |N_s^x|,\\
e_2(x_0, s, t, x, y) = &16 \mathrm{e}^{- \beta |x_0| - \beta x + \beta^2 s} \Phi \Big( - \frac{y}{\sqrt{t-s}} \Big),\\
e_3(x_0, s, t, x, y) = &\mathrm{e}^{- \beta |x_0| - \beta y + \frac{\beta^2}{2}t} \Phi \Big( \frac{y-x}{\sqrt{t-s}} - \beta \sqrt{t-s}\Big),\\
e_4(x_0, s, t, x, y) = &\Phi \Big( - \frac{x-x_0}{\sqrt{s}} \Big). 
\end{align*}
\end{Proposition}
In our applications of this result (see inequalities \eqref{condition2} and \eqref{upper_bound_prop_eq2} below) $x_0$, $s$, $t$, $x$ and $y$ will be chosen in such a way that the first term on the right hand side of \eqref{second_moment_ineq3}  will be $\approx 24 \mathrm{e}^{\beta |x_0|} E^{x_0} |N_s^x| \ E^{x_0} |N_t^y|$ while all of the $e_i$ terms will give a negligible contribution.
\begin{proof}
Using the Markov property and inequality \eqref{eq_expectation00} we obtain 
\begin{align}
\label{bound}
E^{x_0} \big[ |N_s^x| |N_t^y| \big] = &E^{x_0} \Big[ |N_s^x| \ E \big( |N_t^y| \ \big\vert \mathcal{F}_s\big) \Big]\nonumber\\
= &E^{x_0} \Big[ |N_s^x| \sum_{u \in N_s} E^{X^u_s} |N_{t - s}^y| \Big]\nonumber\\
\leq &E^{x_0} \Big[ |N_s^x| \sum_{u \in N_s} \Big( \mathrm{e}^{- \beta |X^u_s| - \beta y + \frac{\beta^2}{2}(t - s)} + \Phi \Big( - \frac{y - X^u_s}{\sqrt{t - s}} \Big) \Big) \Big]\nonumber\\
\leq &\mathrm{e}^{- \beta y + \frac{\beta^2}{2}(t - s)} E^{x_0} \big[ |N_s^x| |N_s| \big] + E^{x_0} \Big[ |N_s^x| \sum_{u \in N_s} \Phi \Big( - \frac{y - X^u_s}{\sqrt{t - s}} \Big) \Big].
\end{align}
\end{proof}
Using Proposition \ref{second_moment_prop2} we bound the first term in \eqref{bound} as
\[
(I) := \mathrm{e}^{- \beta y + \frac{\beta^2}{2}(t - s)} E^{x_0} \big[ |N_s^x| |N_s| \big] \leq (IA) + (IB),
\]
where
\begin{equation}
\label{bound1a}
(IA) = \mathrm{e}^{- \beta y + \frac{\beta^2}{2}(t - s)} E^{x_0} |N_s^x| 
\end{equation}
and
\begin{equation}
\label{bound1b}
(IB) = 16 \mathrm{e}^{- \beta |x_0| - \beta x - \beta y + \frac{\beta^2}{2}s + \frac{\beta^2}{2}t}.
\end{equation}
To treat the second term in \eqref{bound} we first rewrite it using Fubini's Theorem as 
\begin{align*}
E^{x_0} \Big[ |N_s^x| \sum_{u \in N_s} \Phi \Big( - \frac{y - X^u_s}{\sqrt{t - s}} \Big) \Big] = 
&E^{x_0} \Big[ |N_s^x| \sum_{u \in N_s} \Big( \int_{\mathbb{R}} \mathbf{1}_{\big\{z \leq - \frac{y - X^u_s}{\sqrt{t-s}} \big\}} \frac{1}{\sqrt{2 \pi}} \mathrm{e}^{- \frac{z^2}{2}} \mathrm{d}z \Big) \Big]\\
= &\int_{\mathbb{R}} \Big( E^{x_0} \big[ |N_s^x| |N_s^{y + z \sqrt{t-s}}| \big] \Big) \frac{1}{\sqrt{2 \pi}} \mathrm{e}^{- \frac{z^2}{2}} \mathrm{d}z. 
\end{align*}
To estimate this expression we split the integration region into $\{z : y + z \sqrt{t-s} \leq 0\}$ and $\{z : y + z \sqrt{t-s} > 0\}$ so that the second term in \eqref{bound} may be decomposed as 
\[
E^{x_0} \Big[ |N_s^x| \sum_{u \in N_s} \Phi \Big( - \frac{y - X^u_s}{\sqrt{t - s}} \Big) \Big] = (II) + (III),
\]
where $(II)$ and $(III)$ are treated separately below. Firstly, 
\begin{align*}
(II) = &\int_{- \infty}^{- \frac{y}{\sqrt{t-s}}} \Big( E^{x_0} \big[ |N_s^x| |N_s^{y + z \sqrt{t-s}}| \big] \Big) \frac{1}{\sqrt{2 \pi}} \mathrm{e}^{- \frac{z^2}{2}} \mathrm{d}z\\
\leq &\int_{- \infty}^{- \frac{y}{\sqrt{t-s}}} \Big( E^{x_0} \big[ |N_s^x| |N_s| \big] \Big) \frac{1}{\sqrt{2 \pi}} \mathrm{e}^{- \frac{z^2}{2}} \mathrm{d}z\\
= &\Phi \Big( - \frac{y}{\sqrt{t-s}} \Big) E^{x_0} \big[ |N_s^x| |N_s| \big]\\
\leq &(IIA) + (IIB),
\end{align*}
where $(IIA)$ and $(IIB)$ are derived from an application of Proposition \ref{second_moment_prop2} and are
\begin{equation}
\label{bound2a}
(IIA) = \Phi \Big( - \frac{y}{\sqrt{t-s}} \Big) E^{x_0} |N_s^x| 
\end{equation}
and
\begin{equation}
\label{bound2b}
(IIB) = 16 \mathrm{e}^{- \beta |x_0| - \beta y + \beta^2 s} \Phi \Big( - \frac{y}{\sqrt{t-s}} \Big). 
\end{equation}
Also,
\begin{align}
\label{bound3}
(III) = &\int_{- \frac{y}{\sqrt{t-s}}}^\infty \Big( E^{x_0} \big[ |N_s^x| |N_s^{y + z \sqrt{t-s}}| \big] \Big) \frac{1}{\sqrt{2 \pi}} \mathrm{e}^{- \frac{z^2}{2}} \mathrm{d}z \nonumber\\
= &\int_{- \frac{y}{\sqrt{t-s}}}^{- \frac{y - x}{\sqrt{t-s}}} \Big( E^{x_0} |N_s^x| \Big) \frac{1}{\sqrt{2 \pi}} \mathrm{e}^{- \frac{z^2}{2}} \mathrm{d}z +  \int_{- \frac{y-x}{\sqrt{t-s}}}^{\infty} \Big( E^{x_0} |N_s^{y + z \sqrt{t-s}}| \Big) \frac{1}{\sqrt{2 \pi}} \mathrm{e}^{- \frac{z^2}{2}} \mathrm{d}z \nonumber\\
& + \int_{- \frac{y}{\sqrt{t-s}}}^\infty \bigg( 2 \tilde{E}^{x_0} \Big[ \int_0^s E |N_{s - \tau}^x| E |N_{s - \tau}^{y + z \sqrt{t-s}}| \ \mathrm{d} \big( \mathrm{e}^{\beta \tilde{L}_\tau} \big) \Big] \bigg)
\frac{1}{\sqrt{2 \pi}} \mathrm{e}^{- \frac{z^2}{2}} \mathrm{d}z
\end{align}
using Lemma \ref{many_to_two} with $f(\cdot) = \mathbf{1}_{[x, \infty)}(\cdot)$, $g(\cdot) = \mathbf{1}_{[y + z\sqrt{t-s}, \infty)}(\cdot)$ and noting that $f(\cdot)g(\cdot) = f(\cdot)$ if $z \in \big[ - \frac{y}{\sqrt{t-s}}, \ - \frac{y - x}{\sqrt{t-s}}\big]$ and $f(\cdot)g(\cdot) = g(\cdot)$ if $z \in \big[ - \frac{y - x}{\sqrt{t-s}}, \ \infty\big]$. 

Then the first term of \eqref{bound3} is
\begin{equation}
\label{bound3a}
\int_{- \frac{y}{\sqrt{t-s}}}^{- \frac{y - x}{\sqrt{t-s}}} \Big( E^{x_0} |N_s^x| \Big) \frac{1}{\sqrt{2 \pi}} 
\mathrm{e}^{- \frac{z^2}{2}} \mathrm{d}z = \bigg( \Phi \Big( - \frac{y - x}{\sqrt{t-s}} \Big) - \Phi \Big( - \frac{y}{\sqrt{t-s}} \Big) \bigg)E^{x_0} |N_s^x| =: (IIIA). 
\end{equation}
Using inequality \eqref{eq_expectation00} we may bound the second term of \eqref{bound3} as
\begin{align}
\label{bound3bc}
&\int_{- \frac{y-x}{\sqrt{t-s}}}^{\infty} \Big( E^{x_0} |N_s^{y + z \sqrt{t-s}}| \Big) \frac{1}{\sqrt{2 \pi}} \mathrm{e}^{- \frac{z^2}{2}} \mathrm{d}z \nonumber\\ 
\leq &\int_{- \frac{y-x}{\sqrt{t-s}}}^{\infty} \bigg( \mathrm{e}^{- \beta |x_0| - \beta (y + z \sqrt{t-s}) + \frac{\beta^2}{2}s} + \Phi \Big( - \frac{y + z \sqrt{t-s} - x_0}{\sqrt{s}} \Big) \bigg)
\frac{1}{\sqrt{2 \pi}} \mathrm{e}^{- \frac{z^2}{2}} \mathrm{d}z \nonumber\\
\leq &\mathrm{e}^{- \beta |x_0| - \beta y + \frac{\beta^2}{2}t} \int_{- \frac{y-x}{\sqrt{t-s}}}^{\infty} \frac{1}{\sqrt{2 \pi}} \mathrm{e}^{- \frac{1}{2}(z + \beta \sqrt{t-s})^2} \mathrm{d}z + \int_{- \frac{y-x}{\sqrt{t-s}}}^{\infty} \Phi \Big( - \frac{x - x_0}{\sqrt{s}} \Big) \frac{1}{\sqrt{2 \pi}} \mathrm{e}^{- \frac{z^2}{2}} \mathrm{d}z \nonumber\\
= &\mathrm{e}^{- \beta |x_0| - \beta y + \frac{\beta^2}{2}t} \Phi \Big( \frac{y-x}{\sqrt{t-s}} - \beta \sqrt{t-s} \Big) + \Phi \Big( - \frac{x - x_0}{\sqrt{s}} \Big) =: (IIIB) + (IIIC).
\end{align}
Using inequalities \eqref{eq_expectation0} and \eqref{eq_expectation2} we may bound the third term of \eqref{bound3} as
\begin{align*}
&\int_{- \frac{y}{\sqrt{t-s}}}^\infty \bigg( 2 \tilde{E}^{x_0} \Big[ \int_0^s E |N_{s - \tau}^x| E |N_{s - \tau}^{y + z \sqrt{t-s}}| \ \mathrm{d} \big( \mathrm{e}^{\beta \tilde{L}_\tau} \big) \Big] \bigg)
\frac{1}{\sqrt{2 \pi}} \mathrm{e}^{- \frac{z^2}{2}} \mathrm{d}z\\
\leq &\int_{- \frac{y}{\sqrt{t-s}}}^\infty \bigg( 2 \tilde{E}^{x_0} 
\Big[ \int_0^s \mathrm{e}^{- \beta x + \frac{\beta^2}{2}(s - \tau)} 
\mathrm{e}^{- \beta(y + z \sqrt{t-s}) + \frac{\beta^2}{2}(s - \tau)} \mathrm{d} \big( \mathrm{e}^{\beta \tilde{L}_\tau} \big) \Big] \bigg)\frac{1}{\sqrt{2 \pi}} \mathrm{e}^{- \frac{z^2}{2}} \mathrm{d}z\\
\leq &\int_{- \frac{y}{\sqrt{t-s}}}^\infty \bigg( 8 \mathrm{e}^{- \beta |x_0| - \beta x - \beta y + \beta^2 s - \beta z \sqrt{t-s}} \bigg)\frac{1}{\sqrt{2 \pi}} 
\mathrm{e}^{- \frac{z^2}{2}} \mathrm{d}z\\
= &8 \mathrm{e}^{- \beta |x_0| - \beta x - \beta y + \frac{\beta^2}{2}s + \frac{\beta^2}{2}t} \int_{- \frac{y}{\sqrt{t-s}}}^\infty \frac{1}{\sqrt{2 \pi}} \mathrm{e}^{- \frac{1}{2}(z + \beta \sqrt{t - s})^2} \mathrm{d}z\\
\leq &8 \mathrm{e}^{- \beta |x_0| - \beta x - \beta y + \frac{\beta^2}{2}s + \frac{\beta^2}{2}t} =: (IIID).
\end{align*}
Thus we have established inequality \eqref{second_moment_ineq3} with $(IB) + (IIID) = 24 \mathrm{e}^{- \beta |x_0| - \beta x - \beta y + \frac{\beta^2}{2}s + \frac{\beta^2}{2}t}$, $(IA) + (IIA) + (IIIA) = e_1(x_0, s, t, x, y)$, $(IIB) = e_2(x_0, s, t, x, y)$, $(IIIB) = e_3(x_0, s, t, x, y)$ and $(IIIC) = e_4(x_0, s, t, x, y)$.

\section{Proof of the main result}

Let us first note that it is sufficient to prove Theorem \ref{main} for only a single value of $x_0$, which we shall take to be $0$.

Indeed, it can be seen that for any $x$, $y \in \mathbb{R}$, a branching process initiated from level $x$ at time $0$ will hit level $y$ in some almost surely finite time $T$ and thus will contain a branching process initiated from level $y$ at time $T$ by the Strong Markov property. So, for example, 
\[
P^x \Big( \limsup_{t \to \infty} \frac{R_t - \frac{\beta}{2}t}{\log t} \leq \frac{1}{\beta} \Big) 
\leq P^y \Big( \limsup_{t \to \infty} \frac{R_t - \frac{\beta}{2}t}{\log t} \leq \frac{1}{\beta} \Big)
\]
and since $x$ and $y$ are arbitrary it follows that 
\[
P^x \Big( \limsup_{t \to \infty} \frac{R_t - \frac{\beta}{2}t}{\log t} \leq \frac{1}{\beta} \Big) 
= P^y \Big( \limsup_{t \to \infty} \frac{R_t - \frac{\beta}{2}t}{\log t} \leq \frac{1}{\beta} \Big).
\]

\subsection{Upper bound for \eqref{main_limsup}}

\begin{Proposition}
\begin{equation}
\label{main_limsup_upper}
\limsup_{t \to \infty} \frac{R_t - \frac{\beta}{2}t}{\log t} \leq \frac{1}{\beta} \qquad P \text{-a.s.}
\end{equation}
\end{Proposition}
\begin{proof}
For $n \in \mathbb{N}$ and $\epsilon > 0$ let us define the set of particles
\[
\hat{N}^\epsilon_n := \Big\{ u \in N_{n+1} : \sup_{s \in [n, n+1]} X^u_s \geq \frac{\beta}{2}n + \big(\frac{1}{\beta} + \epsilon \big)\log n \Big\}.
\]
Then 
\[
\big\{ R_t > \frac{\beta}{2}t + \big( \frac{1}{\beta} + \epsilon \big) \log t \text{ for some } t \in [n, n+1] \big\} \ \subseteq \ \big\{ |\hat{N}^\epsilon_n| > 0 \big\}
\]
and \eqref{main_limsup_upper} will follow if we can show that for all $\epsilon > 0$
\[
P \big( \big\{ |\hat{N}^\epsilon_n| > 0 \big\} \text{ i.o.}\big) = 0.
\]
By Borel-Cantelli lemma it is sufficient to verify that 
\begin{equation}
\label{finite_sum}
\sum_{n = 0}^\infty P \big( |\hat{N}^\epsilon_n| > 0 \big) < \infty.
\end{equation}
Let us take an arbitrary $\epsilon > 0$. From the Markov's inequality and the Many-to-one lemma (Lemma \ref{many_to_one}) we get that for all $n \geq 0$
\begin{align*}
P \big( |\hat{N}^\epsilon_n| > 0\big) \leq E |\hat{N}^\epsilon_n| 
= &E \Big[ \sum_{u \in N_{n+1}} \mathbf{1}_{\big\{ \sup_{s \in [n, n+1]} X^u_s \geq \frac{\beta}{2}n + (\frac{1}{\beta} + \epsilon )\log n \big\}} \Big]\\
= &\tilde{E} \Big[ \mathrm{e}^{\beta \tilde{L}_{n+1}} \mathbf{1}_{\big\{ \sup_{s \in [n, n+1]} \xi_s \geq \frac{\beta}{2}n + (\frac{1}{\beta} + \epsilon )\log n \big\}} \Big].
\end{align*}
To estimate the latter expectation we split it according to the events $\{ \inf_{s \in [n, n+1]} \xi_s \leq 0 \}$ and $\{ \inf_{s \in [n , n+1]} \xi_s > 0 \}$.

If we let $\xi^{\sup}_{n+1} := \sup_{s \in [n, n+1]} (\xi_s - \xi_n)$ and $\xi^{\inf}_{n+1} := \inf_{s \in [n, n+1]} (\xi_s - \xi_n)$ then 
\begin{align*}
(I) := &\tilde{E} \Big[ \mathrm{e}^{\beta \tilde{L}_{n+1}} \mathbf{1}_{\big\{ \sup_{s \in [n, n+1]} \xi_s \geq \frac{\beta}{2}n + (\frac{1}{\beta} + \epsilon )\log n \big\}} \mathbf{1}_{\big\{ \inf_{s \in [n, n+1]} \xi_s \leq 0 \big\}} \Big]\\
= &\tilde{E} \Big[ \mathrm{e}^{\beta \tilde{L}_{n+1}} \mathbf{1}_{\big\{ \xi_n + \xi^{\sup}_{n+1} \geq \frac{\beta}{2}n + (\frac{1}{\beta} + \epsilon )\log n \big\}} \mathbf{1}_{\big\{ \xi_n + \xi^{\inf}_{n+1} \leq 0 \big\}} \Big]\\
\leq &\tilde{E} \Big[ \mathrm{e}^{\beta \tilde{L}_{n+1}} \mathbf{1}_{\big\{ \xi^{\sup}_{n+1} - \xi^{\inf}_{n+1} \geq \frac{\beta}{2}n \big\}} \Big]\\
\leq &\Big( \tilde{E} \mathrm{e}^{2\beta \tilde{L}_{n+1}} \Big)^{\frac{1}{2}} \Big( \tilde{P} \big( \xi^{\sup}_{n+1} - \xi^{\inf}_{n+1} \geq \frac{\beta}{2}n \big)\Big)^{\frac{1}{2}}
\end{align*}
using the Cauchy-Schwarz inequality in the last line. We know from \eqref{exp_population} that 
\[
\Big( \tilde{E} \mathrm{e}^{2\beta \tilde{L}_{n+1}} \Big)^{\frac{1}{2}} \sim \sqrt{2} \mathrm{e}^{ \beta^2 (n+1)} \quad \text{ as } n \to \infty
\]
while
\begin{align*}
\tilde{P} \big( \xi^{\sup}_{n+1} - \xi^{\inf}_{n+1} \geq \frac{\beta}{2}n \big) \leq 
 &\tilde{P} \big( \xi^{\sup}_{n+1} \geq \frac{\beta}{4}n \big)
+ \tilde{P} \big( - \xi^{\inf}_{n+1} \geq \frac{\beta}{4}n \big)\\
= &2 \mathbb{P} \big( |\mathcal{N}(0, 1)| \geq \frac{\beta}{4} n\big)\\
= &4 \mathbb{P} \big( \mathcal{N}(0, 1) \geq \frac{\beta}{4} n\big)\\
\sim &4 \frac{4}{\sqrt{2 \pi} \beta n} \mathrm{e}^{- \frac{\beta^2}{32} n^2} \quad \text{ as } n \to \infty,
\end{align*}
where in the last line we have used the standard estimate
\begin{equation}
\label{normal_tail}
\mathbb{P} \big( \mathcal{N}(0, 1) \geq x \big) \sim \frac{1}{\sqrt{2 \pi}} \frac{1}{x} \mathrm{e}^{- \frac{x^2}{2}}.
\end{equation}
Thus $(I)$ decays to $0$ at a faster than exponential rate. Also, since $\tilde{L}_{n+1} = \tilde{L}_n$ on the event $\{ \inf_{s \in [n , n+1]} \xi_s > 0 \}$, we have that
\begin{align*}
(II) := &\tilde{E} \Big[ \mathrm{e}^{\beta \tilde{L}_{n+1}} \mathbf{1}_{\big\{ \sup_{s \in [n, n+1]} \xi_s \geq \frac{\beta}{2}n + (\frac{1}{\beta} + \epsilon )\log n \big\}} 
\mathbf{1}_{\big\{ \inf_{s \in [n, n+1]} \xi_s > 0 \big\}} \Big]\\
= &\tilde{E} \Big[ \mathrm{e}^{\beta \tilde{L}_n} \mathbf{1}_{\big\{ \xi_n + \xi^{\sup}_{n+1} \geq \frac{\beta}{2}n + (\frac{1}{\beta} + \epsilon )\log n \big\}} \mathbf{1}_{\big\{ \xi_n + \xi^{\inf}_{n+1} > 0 \big\}} \Big]\\
\leq &\tilde{E} \Big[ \mathrm{e}^{\beta \tilde{L}_n} \mathbf{1}_{\big\{ \xi_n \geq \frac{\beta}{2}n + (\frac{1}{\beta} + \epsilon )\log n - \xi^{\sup}_{n+1} \big\}} \Big].
\end{align*} 
To estimate the latter expectation we split it according to the events $\{ \xi^{\sup}_{n+1} < \frac{\beta}{2} n \}$ and $\{ \xi^{\sup}_{n+1} \geq \frac{\beta}{2} n \}$. Then
\begin{align*}
(IIA) := &\tilde{E} \Big[ \mathrm{e}^{\beta \tilde{L}_n} \mathbf{1}_{\big\{ \xi_n \geq \frac{\beta}{2}n + (\frac{1}{\beta} + \epsilon )\log n - \xi^{\sup}_{n+1} \big\}} \mathbf{1}_{\big\{ \xi^{\sup}_{n+1} \geq \frac{\beta}{2} n \big\}} \Big]\\
\leq &\tilde{E} \Big[ \mathrm{e}^{\beta \tilde{L}_n} \mathbf{1}_{\big\{ \xi^{\sup}_{n+1} \geq \frac{\beta}{2} n \big\}} \Big]\\
\leq &\Big( \tilde{E} \mathrm{e}^{2\beta \tilde{L}_n} \Big)^{\frac{1}{2}} \Big( \tilde{P} \big( \xi^{\sup}_{n+1} \geq \frac{\beta}{2}n \big)\Big)^{\frac{1}{2}},
\end{align*} 
which decays to $0$ at a faster than exponential rate just as in the above calculation. Finally,
\begin{align*}
(IIB) := &\tilde{E} \Big[ \mathrm{e}^{\beta \tilde{L}_n} \mathbf{1}_{\big\{ \xi_n \geq \frac{\beta}{2}n + (\frac{1}{\beta} + \epsilon )\log n - \xi^{\sup}_{n+1} \big\}} \mathbf{1}_{\big\{ \xi^{\sup}_{n+1} < \frac{\beta}{2} n \big\}} \Big]\\
= &\tilde{E} \Big[ \mathbf{1}_{\big\{ \xi^{\sup}_{n+1} < \frac{\beta}{2} n \big\}} \tilde{E} \Big( \mathrm{e}^{\beta \tilde{L}_n} \mathbf{1}_{\big\{ \xi_n \geq \frac{\beta}{2}n + (\frac{1}{\beta} + \epsilon )\log n - \xi^{\sup}_{n+1} \big\}}  \Big\vert \big( \xi_s - \xi_n\big)_{s \in [n, n+1]} \Big) \Big]\\
\leq &\tilde{E} \Big[ \mathbf{1}_{\big\{ \xi^{\sup}_{n+1} < \frac{\beta}{2} n \big\}} \mathrm{e}^{-\beta \big( \frac{\beta}{2}n + (\frac{1}{\beta} + \epsilon )\log n - \xi^{\sup}_{n+1} \big) + \frac{\beta^2}{2}n} \Big]\\
\leq &\frac{C}{n^{1 + \beta \epsilon}},  
\end{align*} 
where we have used \eqref{eq_expectation0} to give an upper bound on the conditional expectation in the second line and where $C = \tilde{E} \mathrm{e}^{\beta \xi^{\sup}_{n+1}} = \mathbb{E} \mathrm{e}^{\beta |\mathcal{N}(0,1)|}$. Thus 
\[
P \big( |\hat{N}^\epsilon_n| > 0 \big) \leq (I) + (IIA) + (IIB)
\]
decays sufficiently fast for \eqref{finite_sum} to be true.
\end{proof}

\subsection{Lower bound for \eqref{main_limsup}}
Before we proceed with the main result of this subsection let us establish the following simple $0$-$1$ law.
\begin{Proposition}
\label{zero_one}
For any $c > 0$ it is true that
\begin{equation}
\label{zero_one_eq}
P \Big( \limsup_{t \to \infty} \frac{R_t - \frac{\beta}{2}t}{\log t} < c \Big) \in \{0, 1\}.
\end{equation}
\end{Proposition}
\begin{proof}
Let $T$ be the time of the first branching and $R^{(1)}_t$, $R^{(2)}_t$, $t \geq 0$ the positions of the rightmost particles of the two subtrees initiated at time $T$. Then since $T$ is almost surely finite and $R^{(1)}_t$ and $R^{(2)}_t$, $t \geq 0$ are independent copies of $R_t$, $t \geq 0$ it follows that
\begin{align*}
P \Big( \limsup_{t \to \infty} \frac{R_t - \frac{\beta}2{}t}{\log t} < c \Big) 
= &P \Big( \limsup_{t \to \infty} \frac{R_{t+T} - \frac{\beta}{2}(t+T)}{\log (t+T)} < c \Big)\\
= &P \Big( \limsup_{t \to \infty} \frac{R_{t+T} - \frac{\beta}{2}t}{\log t} < c \Big)\\
= &P \Big( \limsup_{t \to \infty} \frac{R^{(1)}_t - \frac{\beta}2{}t}{\log t} < c , \
\limsup_{t \to \infty} \frac{R^{(2)}_t - \frac{\beta}2{}t}{\log t} < c \Big)\\
 = &P \Big( \limsup_{t \to \infty} \frac{R_t - \frac{\beta}2{}t}{\log t} < c \Big)^2. 
\end{align*}
Therefore
\[
P \Big( \limsup_{t \to \infty} \frac{R_t - \frac{\beta}{2}t}{\log t} < c \Big) \in \{0, 1\}.
\]
\end{proof}
The following proposition is the main result of this subsection
\begin{Proposition}
\label{prop_limsup_lower}
\begin{equation}
\label{main_limsup_lower}
\limsup_{t \to \infty} \frac{R_t - \frac{\beta}{2}t}{\log t} \geq \frac{1}{\beta} \qquad P \text{-a.s.}
\end{equation}
\end{Proposition}
\begin{proof}
In view of Proposition \ref{zero_one} it is sufficient to show that for all $\epsilon \in (0, \frac{1}{\beta})$ 
\begin{equation}
\label{main_limsup_lower1}
P \Big( \limsup_{t \to \infty} \frac{R_t - \frac{\beta}{2}t}{\log t} \geq \frac{1}{\beta} - \epsilon \Big) > 0.
\end{equation}
Let us take any such $\epsilon$ and fix it for the rest of the proof. If $(t_n)_{n \geq 1}$ is a deterministic sequence such that $t_n \nearrow \infty$ as $n \to \infty$ and 
\[
x_n := \frac{\beta}{2} t_n + (\frac{1}{\beta} - \epsilon) \log t_n \quad \text{ , }  n \geq 1
\]
then we have that
\begin{align*}
P \Big( \limsup_{t \to \infty} \frac{R_t - \frac{\beta}{2}t}{\log t} \geq \frac{1}{\beta} - \epsilon \Big) 
\geq &P \Big( \limsup_{n \to \infty} \frac{R_{t_n} - \frac{\beta}{2}t_n}{\log t_n} \geq \frac{1}{\beta} - \epsilon \Big)\\
= &P \Big( \big\{ \big\vert N_{t_n}^{x_n} \big\vert > 0 \big\} \text{ i.o. } \Big)\\
= &\lim_{n \to \infty} P \Big( \bigcup_{k \geq n} \big\{ \big\vert N_{t_k}^{x_k} \big\vert > 0 \big\} \Big)\\
\geq &\limsup_{n \to \infty} P \Big( \bigcup_{k = n}^{2n} \big\{ \big\vert N_{t_k}^{x_k} \big\vert > 0 \big\} \Big).
\end{align*}
It further follows from Paley-Zygmund inequality that
\begin{align*}
P \Big( \bigcup_{k = n}^{2n} \big\{ \big\vert N_{t_k}^{x_k} \big\vert > 0 \big\} \Big) = &P \Big( \sum_{k = n}^{2n} \big\vert N_{t_k}^{x_k} \big\vert > 0 \Big)\\
\geq &\frac{\Big( E \Big[ \sum_{k = n}^{2n} \big\vert N_{t_k}^{x_k} \big\vert \Big] \Big)^2}{E \Big[ \Big( \sum_{k = n}^{2n} \big\vert N_{t_k}^{x_k} \big\vert \Big)^2 \Big]}\\
= &\frac{\Big( \sum_{k = n}^{2n} E \big\vert N_{t_k}^{x_k} \big\vert \Big)^2}{\sum_{k=n}^{2n} E \Big[ \big\vert N_{t_k}^{x_k} \big\vert^2 \Big] + 2 \sum_{n \leq k < l \leq 2n} E \Big[ \big\vert N_{t_k}^{x_k} \big\vert \big\vert N_{t_l}^{x_l} \big\vert \Big]}.
\end{align*} 
Suppose that the sequence $(t_n)_{n \geq 1}$ could have been chosen in such a way that
\begin{enumerate}
\item
\begin{equation}
\label{condition1}
\sum_{k=n}^{2n} E \big\vert N_{t_k}^{x_k} \big\vert \to \infty \text{ as } n \to \infty,
\end{equation}
\item for all $n$, $k$ and $l$ such that $n \leq k < l \leq 2n$
\begin{equation}
\label{condition2}
E \Big[ \big\vert N_{t_k}^{x_k} \big\vert \big\vert N_{t_l}^{x_l} \big\vert \Big] \leq \alpha_n E \big\vert N_{t_k}^{x_k} \big\vert E \big\vert N_{t_l}^{x_l} \big\vert
\end{equation}
where $(\alpha_n)_{n \geq 1}$ is some converging sequence such that $\alpha_n \to \alpha \in [1, \infty)$ (equivalently, $E[|N_{t_k}^{x_k}| |N_{t_l}^{x_l}| ] \leq C E |N_{t_k}^{x_k}| E |N_{t_l}^{x_l}|$ for some constant $C \geq 1$ and all $n$ sufficiently large).
\end{enumerate}
Then, as we show below, from \eqref{condition1} it directly follows that
\begin{equation}
\label{condition3}
\sum_{k=n}^{2n} E \Big[ \big\vert N_{t_k}^{x_k} \big\vert^2 \Big] \sim \sum_{k=n}^{2n} E \big\vert N_{t_k}^{x_k} \big\vert 
\text{ as } n \to \infty
\end{equation}
and moreover we would then have that
\begin{align*}
P \Big( \bigcup_{k = n}^{2n} \big\{ \big\vert N_{t_k}^{x_k} \big\vert > 0 \big\} \Big) \geq &\frac{\Big( \sum_{k = n}^{2n} E \big\vert N_{t_k}^{x_k} \big\vert \Big)^2}{\sum_{k=n}^{2n} E \Big[ \big\vert N_{t_k}^{x_k} \big\vert^2 \Big] + 2 \sum_{n \leq k < l \leq 2n} E \Big[ \big\vert N_{t_k}^{x_k} \big\vert \big\vert N_{t_l}^{x_l} \big\vert \Big]}\\
\geq &\frac{\Big(\sum_{k = n}^{2n} E \big\vert N_{t_k}^{x_k} \big\vert \Big)^2}{\sum_{k=n}^{2n} E \Big[ \big\vert N_{t_k}^{x_k} \big\vert^2 \Big] + 2 \alpha_n \sum_{n \leq k < l \leq 2n} E \big\vert N_{t_k}^{x_k} \big\vert E \big\vert N_{t_l}^{x_l} \big\vert}\\
\geq &\frac{\Big( \sum_{k = n}^{2n} E \big\vert N_{t_k}^{x_k} \big\vert \Big)^2}{\sum_{k=n}^{2n} E \Big[ \big\vert N_{t_k}^{x_k} \big\vert^2 \Big] + \alpha_n \Big( E \Big[ \sum_{k = n}^{2n} \big\vert N_{t_k}^{x_k} \big\vert \Big] \Big)^2}\\
\to &\frac{1}{\alpha} \quad \text{ as } n \to \infty.
\end{align*}
This would establish \eqref{main_limsup_lower1} and complete the proof of Proposition \ref{prop_limsup_lower}.

Let us now take 
\[
t_n := n^{1 + \beta \epsilon} \qquad \text{ , } n \geq 1.
\]
Below we are going to verify that \eqref{condition1}, \eqref{condition2} and \eqref{condition3} all hold for this choice of $(t_n)_{n \geq 1}$ (in fact, the choice $t_n = n^{1+\delta}$ for any $\delta \in (0, \frac{\beta \epsilon}{1 - \beta \epsilon})$ works just as fine).

\underline{Proof of \eqref{condition1}:}

Take any $n$ and $k$ such that $n \leq k \leq 2n$. From Proposition \ref{expectation} we know that 
\begin{align}
\label{expectation4}
E \big\vert N_{t_k}^{x_k} \big\vert = &\mathrm{e}^{- \beta x_k + \frac{\beta^2}{2}t_k} \Phi \Big( \frac{\beta t_k - x_k}{\sqrt{t_k}} \Big) \nonumber\\
= &\mathrm{e}^{- \beta x_k + \frac{\beta^2}{2}t_k} \Phi \Big( \frac{\beta}{2} \sqrt{t_k} - \big( \frac{1}{\beta} - \epsilon\big) \frac{\log t_k}{\sqrt{t_k}} \Big) \nonumber\\
\geq &\alpha_n' \mathrm{e}^{- \beta x_k + \frac{\beta^2}{2}t_k},
\end{align}
where
\[
\alpha_n' = \Phi \Big( \frac{\beta}{2} \sqrt{t_n} - \big( \frac{1}{\beta} - \epsilon\big) \frac{\log t_{2n}}{\sqrt{t_n}} \Big) \to 1 
\]
as $n \to \infty$. Hence
\[
\sum_{k=n}^{2n} E \big\vert N_{t_k}^{x_k} \big\vert \geq \alpha_n' \sum_{k=n}^{2n} \mathrm{e}^{- \beta x_k + \frac{\beta^2}{2}t_k} = \alpha_n' \sum_{k=n}^{2n} \frac{1}{(t_k)^{1 - \beta \epsilon}} \geq \alpha_n' \frac{n}{(t_{2n})^{1 - \beta \epsilon}} = \frac{\alpha_n'}{2^{1 - \beta^2 \epsilon^2}} n^{\beta^2 \epsilon^2} \to \infty
\]
as $n \to \infty$.

\underline{Proof of \eqref{condition3}:}

Trivially $E \vert N_{t_k}^{x_k} \vert \leq E \big[ \vert N_{t_k}^{x_k} \vert^2 \big]$ . Also from Proposition \ref{second_moment_prop} and inequality \eqref{expectation4} we know that
\begin{align*}
E \Big[ \big\vert N_{t_k}^{x_k} \big\vert^2 \Big] \leq &E \big\vert N_{t_k}^{x_k} \big\vert + 
8 \mathrm{e}^{- 2 \big( \beta x_k - \frac{\beta^2}{2} t_k \big)}\\ 
\leq &\Big( 1 + \frac{8}{\alpha_n'} \mathrm{e}^{- \beta x_k + \frac{\beta^2}{2} t_k} \Big)E \big\vert N_{t_k}^{x_k} \big\vert\\
\leq &\Big( 1 + \frac{8}{\alpha_n'} \frac{1}{(t_n)^{1 - \beta \epsilon}} \Big)E \big\vert N_{t_k}^{x_k} \big\vert.
\end{align*}
Thus
\[
\sum_{k=n}^{2n} E \big\vert N_{t_k}^{x_k} \big\vert \leq \sum_{k=n}^{2n} E \big[ \big\vert N_{t_k}^{x_k} \big\vert^2 \big] \leq (1 + \alpha_n'' ) \sum_{k=n}^{2n} E \big\vert N_{t_k}^{x_k} \big\vert,
\]
where
\[
\alpha_n'' = \frac{8}{\alpha_n'} \frac{1}{(t_n)^{1 - \beta \epsilon}} \to 0 \quad \text{ as } n \to \infty
\]
and this proves \eqref{condition3}.

\underline{Proof of \eqref{condition2}:}

Let us  take any $k$, $l$ and $n$ such that $n \leq k < l \leq 2n$. From Proposition \ref{second_moment_prop3} and inequality \eqref{expectation4} we have that 
\begin{align*}
E \Big[ \big\vert N_{t_k}^{x_k} \big\vert \big\vert N_{t_l}^{x_l} \big\vert \Big] \leq &24 \mathrm{e}^{- \beta x_k - \beta x_l + \frac{\beta^2}{2}t_k + \frac{\beta^2}{2}t_l} + \sum_{i = 1}^4 e_i(0, t_k, t_l, x_k, x_l)\\
\leq &\frac{24}{(\alpha_n')^2} E \big\vert N_{t_k}^{x_k} \big\vert E \big\vert N_{t_l}^{x_l} \big\vert + \sum_{i = 1}^4 e_i(0, t_k, t_l, x_k, x_l).
\end{align*}
Let us now check that the contribution from the $e_i$ terms is negligible. Firstly, 
\begin{align*}
\mathrm{e}^{\beta x_l - \frac{\beta^2}{2}t_l} \Phi \Big( - \frac{x_l - x_k}{\sqrt{t_l - t_k}} \Big) = &(t_l)^{1 - \beta\epsilon}\Phi \Big( - \frac{\beta}{2} \sqrt{t_l - t_k} - \big( \frac{1}{\beta} - \epsilon\big) 
\frac{\log t_l - \log t_k}{\sqrt{t_l - t_k}} \Big)\\
\leq &(t_l)^{1 - \beta\epsilon}\Phi \Big( - \frac{\beta}{2} \sqrt{t_l - t_k}\Big)\\
\leq &(t_{2n})^{1 - \beta\epsilon}\Phi \Big( - \frac{\beta}{2} \sqrt{t_{n+1} - t_n}\Big)\\
= &(2n)^{1 - \beta^2 \epsilon^2}\Phi \Big( - \frac{\beta}{2} \sqrt{(n+1)^{1 + \beta \epsilon} - n^{1 + \beta \epsilon}}\Big)\\
\leq & 2n \Phi \Big( - \frac{\beta}{2} n^{\frac{\beta \epsilon}{2}}\Big) \quad =: \alpha_n^{'''} \quad (\to 0 \text{ as } n \to \infty)
\end{align*}
and hence
\begin{align*}
e_1(0, t_k, t_l, x_k, x_l) = &\bigg( \mathrm{e}^{-\frac{\beta^2}{2}t_k} + \mathrm{e}^{\beta x_l - \frac{\beta^2}{2} t_l} \Phi \Big( - \frac{x_l-x_k}{\sqrt{t_l-t_k}} \Big) \bigg) \mathrm{e}^{- \beta x_l + \frac{\beta^2}{2} t_l} E^{x_0} \big\vert N_{t_k}^{x_k} \big\vert\\
\leq &\big( \mathrm{e}^{-\frac{\beta^2}{2}t_n} + \alpha_n''' \big) \frac{1}{\alpha_n'} E^{x_0} \big\vert N_{t_k}^{x_k} \big\vert E^{x_0} \big\vert N_{t_l}^{x_l} \big\vert.
\end{align*}
Also,
\begin{align*}
\mathrm{e}^{\beta x_l - \frac{\beta^2}{2}(t_l - t_k)} \Phi \Big( - \frac{x_l}{\sqrt{t_l - t_k}} \Big) \leq  
&\mathrm{e}^{\beta x_l - \frac{\beta^2}{2}(t_l - t_k)} \frac{1}{\sqrt{2 \pi}} \frac{\sqrt{t_l - t_k}}{x_l} \mathrm{e}^{- \frac{1}{2} \frac{x_l^2}{t_l - t_k}}\\
= &\frac{1}{\sqrt{2 \pi}} \frac{\sqrt{t_l - t_k}}{x_l} \mathrm{e}^{- \frac{1}{2(t_l - t_k)} \big(x_l - \beta(t_l - t_k)\big)^2}\\
\leq &\frac{1}{\sqrt{2 \pi}} \frac{\sqrt{t_{2n}}}{x_n}\\
\leq &\frac{\sqrt{2}}{\beta \sqrt{\pi}} \frac{\sqrt{t_{2n}}}{t_n}\\
= &\frac{\sqrt{2^{2 + \beta \epsilon}}}{\beta \sqrt{\pi}} n^{- \frac{1 + \beta \epsilon}{2}} \ =: \alpha_n^{(iv)} \quad (\to 0 \text{ as } n \to \infty)
\end{align*}
and hence 
\begin{align*}
e_2(0, t_k, t_l, x_k, x_l) = &16 \mathrm{e}^{- \beta x_k + \frac{\beta^2}{2}t_k} \mathrm{e}^{- \beta x_l + \frac{\beta^2}{2}t_l} \mathrm{e}^{\beta x_l - \frac{\beta^2}{2}(t_l - t_k)} \Phi \Big( - \frac{x_l}{\sqrt{t_l - t_k}} \Big)\\
\leq & \frac{16 \alpha_n^{(iv)}}{(\alpha_n')^2}  E^{x_0} \big\vert N_{t_k}^{x_k} \big\vert E^{x_0} \big\vert N_{t_l}^{x_l} \big\vert.
\end{align*}
Also,
\begin{align*}
\mathrm{e}^{\beta x_k - \frac{\beta^2}{2}t_k} \Phi \Big( \frac{x_l - x_k}{\sqrt{t_l - t_k}} - \beta \sqrt{t_l - t_k} \Big) 
= &(t_k)^{1 - \beta \epsilon} \Phi \Big( - \frac{\beta}{2} \sqrt{t_l - t_k} + \big( \frac{1}{\beta} - \epsilon \big) 
\frac{\log t_l - \log t_k}{\sqrt{t_l - t_k}} \Big)\\
\leq &(t_{2n})^{1 - \beta \epsilon} \Phi \Big( - \frac{\beta}{2} \sqrt{t_{n+1} - t_n} + \big( \frac{1}{\beta} - \epsilon \big) 
\log t_{2n} \Big)\\
\leq &2 n \Phi \Big( - \frac{\beta}{2} n^{\frac{\beta \epsilon}{2}} + \frac{1}{\beta} \log (2n) \Big) \ =: \alpha_n^{(v)} \ (\to 0 \text{ as } n \to \infty)
\end{align*}
and hence 
\begin{align*}
e_3(0, t_k, t_l, x_k, x_l) = &\mathrm{e}^{- \beta x_k + \frac{\beta^2}{2}t_k} \mathrm{e}^{- \beta x_l + \frac{\beta^2}{2}t_l} \mathrm{e}^{\beta x_k - \frac{\beta^2}{2}t_k} \Phi \Big( \frac{x_l - x_k}{\sqrt{t_l - t_k}} - 
\beta \sqrt{t_l - t_k} \Big)\\ 
\leq &\frac{\alpha_n^{(v)}}{(\alpha_n')^2}  E^{x_0} \big\vert N_{t_k}^{x_k} \big\vert E^{x_0} \big\vert N_{t_l}^{x_l} \big\vert.
\end{align*}
Finally,
\begin{align*}
\mathrm{e}^{\beta x_k - \frac{\beta^2}{2}t_k} \mathrm{e}^{\beta x_l - \frac{\beta^2}{2}t_l} \Phi \Big( - \frac{x_k}{\sqrt{t_k}} \Big) = &(t_k)^{1 - \beta \epsilon} (t_l)^{1 - \beta \epsilon} \Phi \Big( - \frac{\beta}{2} \sqrt{t_k} - \big( \frac{1}{\beta} - \epsilon \big) \frac{\log t_k}{\sqrt{t_k}} \Big)\\ 
\leq &(t_{2n})^{2 - 2\beta \epsilon}\Phi \Big( - \frac{\beta}{2} \sqrt{t_n} \Big)\\ 
\leq &4 n^2 \Phi \Big( - \frac{\beta}{2} \sqrt{n} \Big) \quad =: \alpha_n^{(vi)} \ (\to 0 \text{ as } n \to \infty)
\end{align*}
and hence
\begin{align*}
e_4(0, t_k, t_l, x_k, x_l) = &\mathrm{e}^{\beta x_k - \frac{\beta^2}{2}t_k} \mathrm{e}^{\beta x_l - \frac{\beta^2}{2}t_l} \Phi \Big( - \frac{x_k}{\sqrt{t_k}} \Big) \mathrm{e}^{- \beta x_k + \frac{\beta^2}{2}t_k} \mathrm{e}^{- \beta x_l + \frac{\beta^2}{2}t_l}\\ 
\leq &\frac{\alpha_n^{(vi)}}{(\alpha_n')^2}  E^{x_0} \big\vert N_{t_k}^{x_k} \big\vert E^{x_0} \big\vert N_{t_l}^{x_l} \big\vert.
\end{align*}
\end{proof}

\subsection{Lower bound for \eqref{main_liminf}}

\begin{Proposition}
\label{prop_liminf_lower}
\begin{equation}
\label{main_liminf_lower}
\liminf_{t \to \infty} \frac{R_t - \frac{\beta}{2}t}{\log \log t} \geq - \frac{1}{\beta} \qquad P \text{-a.s.}
\end{equation}
\end{Proposition}
\begin{proof}
For $t \geq 0$, $n \in \{0 , 1, 2, \cdots \}$ and $\epsilon > 0$ let us consider the following subsets of $N_{t+1}$:
\[
\tilde{N}^{n, \epsilon}_t := \Big\{ u \in N_{t+1} : \inf_{s \in [t, t+1]} X^u_s \geq \frac{\beta}{2}(n+1) - \big(\frac{1}{\beta} + \epsilon \big)\log\log (n+1) \Big\}.
\]
Note that 
\[
\big\{ R_t < \frac{\beta}{2}t - \big( \frac{1}{\beta} + \epsilon \big) \log \log t \text{ for some } t \in [n, n+1] \big\} \ \subseteq \ \big\{ |\tilde{N}^{n, \epsilon}_n| = 0 \big\}
\]
and so in order to prove Proposition \ref{main_liminf_lower} it is sufficient to show that for all $\epsilon > 0$, 
\[
P \big( \big\{ |\tilde{N}^{n, \epsilon}_n| = 0 \big\} \text{ i.o.}\big) = 0.
\]

For the rest of the proof let us fix any $\epsilon > 0$ and a deterministic sequence of times $s_n$, $n \geq 1$ such that $n - s_n > 0$ for all $n \geq 1$, $s_n = o(n)$ and $(\log n) \mathrm{e}^{- \frac{\beta^2}{2}s_n} \to 0$ as $n \to \infty$ (e.g., we can take $s_n = \sqrt{n - 1}$, $n \geq 1$).

If we define the sequence of events
\[
\mathcal{A}_n := \Big\{ \big\vert X^u_{s_n} \big\vert \leq \beta s_n \ \forall u \in s_n \ , \ M_{s_n} \geq (\log n)^{-\frac{\beta \epsilon}{2}} \Big\} \in \mathcal{F}_{s_n}, \quad n \geq 1
\]
then, from \eqref{rightmost_as} and symmetry and also from the fact that $M_{s_n}$ converges almost surely to a strictly positive limit, $P (\mathcal{A}_n \text{ ev.}) = 1$. Therefore
\[
P \big( \big\{ \big\vert\tilde{N}^{n, \epsilon}_n\big\vert = 0 \big\} \text{ i.o.}\big) = P \big( \big\{ \{ \big\vert \tilde{N}^{n, \epsilon}_n \big\vert = 0 \}, \  \mathcal{A}_n \big\} \text{ i.o.}\big).
\]
and so it is sufficient for us to verify that 
\begin{equation}
\label{finite_sum_3_3}
\sum_{n = 1}^\infty P \big( \{ \big\vert \tilde{N}^{n, \epsilon}_n \big\vert = 0 \} , \ \mathcal{A}_n \big) < \infty.
\end{equation}
From the Markov property we have that 
\begin{align*}
P \big( \{ \big\vert \tilde{N}^{n,\epsilon}_n \big\vert = 0 \}, \ \mathcal{A}_n \big) = &E \Big[ \mathbf{1}_{\mathcal{A}_n} \prod_{u \in N_{s_n}} P^{X^u_{s_n}} \big( \big\vert \tilde{N}^{n, \epsilon}_{n-s_n} \big\vert = 0 \big) \Big]\\
= &E \Big[ \mathbf{1}_{\mathcal{A}_n} \prod_{u \in N_{s_n}} \Big( 1 - P^{X^u_{s_n}} \big( \big\vert \tilde{N}^{n,\epsilon}_{n-s_n} \big\vert > 0 \big) \Big) \Big]\\
\leq &E \Big[ \mathbf{1}_{\mathcal{A}_n} \exp \Big\{ - \sum_{u \in N_{s_n}} P^{X^u_{s_n}} \big( \big\vert \tilde{N}^{n,\epsilon}_{n-s_n} \big\vert > 0 \big) \Big\} \Big],
\end{align*}
where for the last line we also used the trivial fact that $1 - x \leq \mathrm{e}^{-x}$, $x \in \mathbb{R}$. 

Proposition \ref{pre_liminf_lower}, which we prove below, then says that there exists some strictly positive constant $C$ such that for all $n$ sufficiently large, on the event $\mathcal{A}_n$,
\begin{equation}
\label{pre_ineq}
P^{X^u_{s_n}} \big( \big\vert \tilde{N}^{n,\epsilon}_{n-s_n} \big\vert > 0 \big) \geq C (\log n)^{1 + \beta \epsilon} \mathrm{e}^{- \beta |X^u_{s_n}| - \frac{\beta^2}{2}s_n}.
\end{equation}
\Big( Informally, one may say that 
\begin{align*}
P^{X^u_{s_n}} \big( \big\vert \tilde{N}^{n,\epsilon}_{n-s_n} \big\vert > 0 \big) \approx E^{X^u_{s_n}} \big\vert \tilde{N}^{n,\epsilon}_{n-s_n} \big\vert \approx &E^{X^u_{s_n}} \Big\vert N^{\frac{\beta}{2}(n+1) - (\frac{1}{\beta} + \epsilon )\log\log (n+1)}_{n-s_n+1} \Big\vert\\ 
\approx &(\log n)^{1 + \beta \epsilon} \mathrm{e}^{- \beta |X^u_{s_n}| - \frac{\beta^2}{2}s_n}. \Big)
\end{align*}
Thus
\begin{align*}
E \Big[ \mathbf{1}_{\mathcal{A}_n} \exp \Big\{ - \sum_{u \in N_{s_n}} P^{X^u_{s_n}} \big( \big\vert \tilde{N}^{n,\epsilon}_{n-s_n} \big\vert > 0 \big) \Big\} \Big] 
\leq &E \Big[ \mathbf{1}_{\mathcal{A}_n} \mathrm{e}^{- C M_{s_n} (\log n)^{1 + \beta \epsilon}} \Big]\\
\leq &E \Big[ \mathbf{1}_{\mathcal{A}_n} \mathrm{e}^{- C (\log n)^{1 + \frac{\beta \epsilon}{2}}} \Big]\\
\leq &\Big( \frac{1}{n} \Big)^{C (\log n)^{\frac{\beta \epsilon}{2}}}\\
\leq &\frac{1}{n^2}
\end{align*}
for all $n$ sufficiently large. This proves \eqref{finite_sum_3_3} and consequently Proposition \ref{main_liminf_lower}.
\end{proof}
We complete this subsection with the proof of \eqref{pre_ineq}, which is stated as a separate result to lighten the proof of Proposition \ref{main_liminf_lower}.
\begin{Proposition}
\label{pre_liminf_lower}
Let $\epsilon > 0$ be some constant and $(s_n)_{n \geq 1}$ a sequence such that $n - s_n > 0$ for all $n \geq 1$ and on one hand $s_n = o(n)$ while on the other hand $(\log n) \mathrm{e}^{- \frac{\beta^2}{2}s_n} \to 0$ as $n \to \infty$. 

Then there exists a strictly positive constant $C$ such that for all $x_0$ with $|x_0| \leq \beta s_n$, 
\[
P^{x_0} \big( |\tilde{N}^{n, \epsilon}_{n - s_n}| > 0 \big) \geq C (\log n)^{1 + \beta \epsilon} \mathrm{e}^{- \beta |x_0| - \frac{\beta^2}{2}s_n} 
\]
for all $n$ sufficiently large.
\end{Proposition}
\begin{proof}
Let $\epsilon$, $(s_n)_{n \geq 1}$ and $x_0$ be as above. If we let 
\[
\xi^{\inf}_{n - s_n +1} := \inf_{s \in [n-s_n, n-s_n+1]} (\xi_s - \xi_{n - s_n}) \stackrel{d}{=} - | \mathcal{N}(0,1) |,
\] 
for the spine process $(\xi_t)_{t \geq 0}$ as in Subsection 3.1 then we have from Lemma \ref{many_to_one} 
\begin{align*}
E^{x_0} |\tilde{N}^{n, \epsilon}_{n - s_n}| = &E^{x_0} \Big[ \sum_{u \in N_{n - s_n + 1}} \mathbf{1}_{\big\{ \inf_{s \in [n - s_n, n - s_n +1]} X^u_s \geq \frac{\beta}{2}(n+1) - (\frac{1}{\beta} + \epsilon )\log\log (n+1) \big\}} \Big]\\
= &\tilde{E}^{x_0} \Big[ \mathrm{e}^{\beta \tilde{L}_{n - s_n + 1}} \mathbf{1}_{\big\{ \inf_{s \in [n - s_n, n - s_n +1]} \xi_s \geq \frac{\beta}{2}(n+1) - (\frac{1}{\beta} + \epsilon )\log\log (n+1) \big\}} \Big]\\
= &\tilde{E}^{x_0} \Big[ \mathrm{e}^{\beta \tilde{L}_{n - s_n}} \mathbf{1}_{\big\{ \xi_{n - s_n} + \xi^{\inf}_{n - s_n + 1} \geq \frac{\beta}{2}(n+1) - (\frac{1}{\beta} + \epsilon )\log\log (n+1) \big\}} \Big]
\end{align*}
for all $n$ sufficiently large which guarantees that $\frac{\beta}{2}(n+1) - \big(\frac{1}{\beta} + \epsilon \big)\log\log (n+1) > 0$ so that $\tilde{L}_{n -s_n + 1} = \tilde{L}_{n - s_n}$. 

Then using independence of $(\xi_s)_{s \in [0, n-s_n]}$ and $(\xi_s - \xi_{n-s_n})_{s \in [n-s_n, n-s_n+1]}$ together with identity \eqref{eq_expectation} we get
\begin{align*}
&\tilde{E}^{x_0} \Big[ \mathrm{e}^{\beta \tilde{L}_{n - s_n}} \mathbf{1}_{\big\{ \xi_{n - s_n} + \xi^{\inf}_{n - s_n + 1} \geq \frac{\beta}{2}(n+1) - (\frac{1}{\beta} + \epsilon )\log\log (n+1) \big\}} \Big]\\
\geq &\tilde{E}^{x_0} \Big[ \tilde{E}^{x_0} \Big( \mathrm{e}^{\beta \tilde{L}_{n - s_n}} \mathbf{1}_{\big\{ \xi_{n - s_n}   \geq \frac{\beta}{2}(n+1) - (\frac{1}{\beta} + \epsilon )\log\log (n+1) - \xi^{\inf}_{n - s_n + 1}\big\}} \Big\vert (\xi_s - \xi_{n-s_n})_{s \in [n-s_n, n-s_n+1]} \Big) \Big]\\
\geq &\tilde{E}^{x_0} \Big[ \mathrm{e}^{- \beta |x_0| - \beta \big(\frac{\beta}{2}(n+1) - (\frac{1}{\beta} + \epsilon )\log\log (n+1) - \xi^{\inf}_{n - s_n + 1}\big) + \frac{\beta^2}{2}(n - s_n)}\\ 
&\qquad\qquad\qquad\qquad \times \Phi \Big( \frac{\beta(n-s_n) - |x_0| - \big( \frac{\beta}{2}(n+1) - (\frac{1}{\beta} + \epsilon )\log\log (n+1) - \xi^{\inf}_{n - s_n + 1} \big)}{\sqrt{n - s_n}} \Big) \Big]\\
= &\tilde{E}^{x_0} \Big[ \mathrm{e}^{- \beta |x_0| - \frac{\beta^2}{2}s_n} \big( \log (n+1)\big)^{1 + \beta \epsilon} \mathrm{e}^{- \frac{\beta^2}{2}} \mathrm{e}^{\beta \xi^{\inf}_{n - s_n + 1}}\\
&\qquad\qquad\qquad\qquad \times \Phi \Big( \frac{\frac{\beta}{2}(n-1) - \beta s_n - |x_0| + (\frac{1}{\beta} + \epsilon )\log\log (n+1) + \xi^{\inf}_{n - s_n + 1} \big)}{\sqrt{n - s_n}} \Big)\Big].
\end{align*}
Then recalling that $s_n = o(n)$ and $|x_0| \leq \beta s_n$ we see that for all $n$ sufficiently large the argument of the $\Phi(\cdot)$ function above is $\geq \xi^{\inf}_{n - s_n + 1}$ and hence
\begin{align}
\label{inequality1}
E^{x_0} |\tilde{N}^{n, \epsilon}_{n - s_n}| \geq &\mathrm{e}^{- \beta |x_0| - \frac{\beta^2}{2}s_n} (\log n)^{1 + \beta \epsilon} \mathrm{e}^{- \frac{\beta^2}{2}} \tilde{E}^{x_0} \Big[  \mathrm{e}^{\beta \xi^{\inf}_{n - s_n + 1}} \Phi \big( \xi^{\inf}_{n - s_n + 1} \big) \Big]\nonumber\\
= &C_1 \mathrm{e}^{- \beta |x_0| - \frac{\beta^2}{2}s_n} (\log n)^{1 + \beta \epsilon},
\end{align}
 where
\[
C_1 = \mathrm{e}^{- \frac{\beta^2}{2}} \tilde{E}^{x_0} \Big[  \mathrm{e}^{\beta \xi^{\inf}_{n - s_n + 1}} \Phi \big( \xi^{\inf}_{n - s_n + 1} \big) \Big] = \mathrm{e}^{- \frac{\beta^2}{2}} \int_{\mathbb{R}} \mathrm{e}^{- \beta |x|} \Phi (- |x|) \frac{1}{\sqrt{2 \pi}} \mathrm{e}^{- \frac{x^2}{2}} \mathrm{d}x \in (0, \infty).
\]
Also, since $\tilde{N}^{n, \epsilon}_{n - s_n} \subseteq N_{n-s_n+1}^{\frac{\beta}{2}(n+1) - (\frac{1}{\beta} + \epsilon )\log\log (n+1)}$, we have that
\begin{align}
\label{ineq_a1}
E^{x_0} |\tilde{N}^{n, \epsilon}_{n - s_n}|^2 \leq &E^{x_0} \Big\vert N_{n-s_n+1}^{\frac{\beta}{2}(n+1) - (\frac{1}{\beta} + \epsilon )\log\log (n+1)} \Big\vert^2 \nonumber\\
= &E^{x_0} \Big\vert N_{n-s_n+1}^{\frac{\beta}{2}(n+1) - (\frac{1}{\beta} + \epsilon )\log\log (n+1)} \Big\vert \nonumber\\ 
& \qquad + 2 \tilde{E}^{x_0} \Big[ \int_0^{n-s_n+1} \Big( E \Big\vert N_{n-s_n+1 - \tau}^{\frac{\beta}{2}(n+1) - (\frac{1}{\beta} + \epsilon )\log\log (n+1)} \Big\vert \Big)^2 \mathrm{d} \big( \mathrm{e}^{\beta \tilde{L}_\tau}\big) \Big].
\end{align}
Then from \eqref{eq_expectation00},
\begin{align}
\label{ineq_a2}
E^{x_0} \Big\vert N_{n-s_n+1}^{\frac{\beta}{2}(n+1) - (\frac{1}{\beta} + \epsilon )\log\log (n+1)} \Big\vert 
\leq &\mathrm{e}^{- \beta |x_0| - \beta \big( \frac{\beta}{2}(n+1) - (\frac{1}{\beta} + \epsilon )\log\log (n+1) \big) +\frac{\beta^2}{2}(n - s_n + 1)} \nonumber\\
& \qquad + \Phi \Big( - \frac{\frac{\beta}{2}(n+1) - (\frac{1}{\beta} + \epsilon )\log\log (n+1) - x_0}{\sqrt{n-s_n+1}} \Big) \nonumber\\
\leq &\mathrm{e}^{- \beta |x_0| - \frac{\beta^2}{2}s_n} \big( \log (n+1) \big)^{1 + \beta \epsilon} \nonumber\\
& \qquad + \Phi \Big( - \frac{\beta}{2} \sqrt{n+1} + \frac{ (\frac{1}{\beta} + \epsilon )\log\log (n+1) + \beta s_n}{\sqrt{n-s_n+1}} \Big) \nonumber\\
\leq &C_2 \mathrm{e}^{- \beta |x_0| - \frac{\beta^2}{2}s_n} (\log n)^{1 + \beta \epsilon}
\end{align}
for any constant $C_2 > 1$ and all $n$ large enough. 

Also, for $\tau \in [0, n-s_n+1]$ and in the case $x_0 = 0$ we have from \eqref{eq_expectation0} that
\begin{align}
\label{ineq_a3}
E \Big\vert N_{n-s_n+1 - \tau}^{\frac{\beta}{2}(n+1) - (\frac{1}{\beta} + \epsilon )\log\log (n+1)} \Big\vert 
\leq &\mathrm{e}^{- \beta \big( \frac{\beta}{2}(n+1) - (\frac{1}{\beta} + \epsilon )\log\log (n+1) \big) +\frac{\beta^2}{2}(n - s_n + 1 - \tau)} \nonumber\\
\leq &C_3 \mathrm{e}^{- \frac{\beta^2}{2}s_n} (\log n)^{1 + \beta \epsilon} \mathrm{e}^{- \frac{\beta^2}{2} \tau}
\end{align}
for any constant $C_3 > 1$ and all $n$ large enough. 

Putting together inequalities \eqref{ineq_a1}, \eqref{ineq_a2}, \eqref{ineq_a3} and making use of Proposition \ref{expectation2} we get that 
\begin{align}
\label{ineq_b}
E^{x_0} |\tilde{N}^{n, \epsilon}_{n - s_n}|^2 \leq &C_2 \mathrm{e}^{- \beta |x_0| - \frac{\beta^2}{2}s_n} (\log n)^{1 + \beta \epsilon} \nonumber\\ 
& \qquad + 2 \tilde{E}^{x_0} \Big[ \int_0^{n-s_n+1} \Big( C_3 \mathrm{e}^{- \frac{\beta^2}{2}s_n} (\log n)^{1 + \beta \epsilon} \mathrm{e}^{- \frac{\beta^2}{2} \tau} \Big)^2 \mathrm{d} \big( \mathrm{e}^{\beta \tilde{L}_\tau}\big) \Big]\\
\leq &C_2 \mathrm{e}^{- \beta |x_0| - \frac{\beta^2}{2}s_n} (\log n)^{1 + \beta \epsilon} + 8(C_3)^2 \mathrm{e}^{- \beta |x_0| - \frac{\beta^2}{2} s_n} (\log n)^{1 + \beta \epsilon} \big( \mathrm{e}^{- \frac{\beta^2}{2} s_n} (\log n)^{1 + \beta \epsilon} \big)\nonumber\\
\leq & C_4 \mathrm{e}^{- \beta |x_0| - \frac{\beta^2}{2}s_n} (\log n)^{1 + \beta \epsilon}
\end{align}
for any constant $C_4 > C_2$ and $n$ large enough. It then follows from inequalities \eqref{inequality1}, \eqref{ineq_b} and Paley-Zygmund inequality that for all $n$ large enough
\[
P^{x_0} \big( |\tilde{N}^{n, \epsilon}_{n - s_n}| > 0 \big) \geq \frac{(C_1)^2}{C_4} \mathrm{e}^{- \beta |x_0| - \frac{\beta^2}{2}s_n} (\log n)^{1 + \beta \epsilon},
\]
which proves the result.
\end{proof}

\subsection{Upper bound for \eqref{main_liminf}}

\begin{Proposition}
\label{prop_liminf_upper}
\begin{equation}
\label{main_liminf_upper}
\liminf_{t \to \infty} \frac{R_t - \frac{\beta}{2}t}{\log \log t} \leq - \frac{1}{\beta} \qquad P \text{-a.s.}
\end{equation}
\end{Proposition}
\begin{proof}
It is sufficient to prove that for any $\epsilon \in (0, \frac{1}{\beta})$ 
\begin{equation}
\label{main_liminf_upper1}
\liminf_{t \to \infty} \frac{R_t - \frac{\beta}{2}t}{\log \log t} \leq - \frac{1}{\beta} + \epsilon \quad P \text{-a.s.}
\end{equation}
Let us take any such $\epsilon$ and fix it for the rest of the proof. If $(t_n)_{n \geq 1}$ is a deterministic sequence such that $t_n \nearrow \infty$ as $n \to \infty$ and 
\[
y_n = \frac{\beta}{2} t_n - (\frac{1}{\beta} - \epsilon) \log \log t_n \quad \text{ , }  n \geq 1
\]
then
\begin{align*}
P \Big( \liminf_{t \to \infty} \frac{R_t - \frac{\beta}{2}t}{\log \log t} \leq - \frac{1}{\beta} + \epsilon \Big) 
\geq &P \Big( \liminf_{n \to \infty} \frac{R_{t_n} - \frac{\beta}{2}t_n}{\log \log t_n} \leq - \frac{1}{\beta} + \epsilon \Big)\\
= &P \Big( \big\{ \big\vert N_{t_n}^{y_n} \big\vert = 0 \big\} \text{ i.o. } \Big)\\
= &\lim_{n \to \infty} P \Big( \bigcup_{k \geq n} \big\{ \big\vert N_{t_k}^{y_k} \big\vert = 0 \big\} \Big)\\
\geq &\limsup_{n \to \infty} P \Big( \bigcup_{k = n}^{2n} \big\{ \big\vert N_{t_k}^{y_k} \big\vert = 0 \big\} \Big).
\end{align*}
Then for any time $s_n < t_n$ and event $\mathcal{A}_n \in \mathcal{F}_{s_n}$ we further have that
\begin{align}
\label{main_liminf_upper2}
P \Big( \bigcup_{k = n}^{2n} \big\{ \big\vert N_{t_k}^{y_k} \big\vert = 0 \big\} \Big) = 
&P \Big( \sum_{k=n}^{2n} \mathbf{1}_{\{|N_{t_k}^{y_k}| = 0\}} > 0 \Big)\nonumber\\
\geq &E \Big[ \mathbf{1}_{\mathcal{A}_n} P \Big( \sum_{k=n}^{2n} \mathbf{1}_{\{|N_{t_k}^{y_k}| = 0\}} > 0 \ \big\vert \mathcal{F}_{s_n} \Big) \Big].
\end{align}
From conditional Cauchy-Schwarz inequality applied to
\[
\sum_{k=n}^{2n} \mathbf{1}_{\{|N_{t_k}^{y_k}| = 0\}} = \bigg( \sum_{k=n}^{2n} \mathbf{1}_{\{|N_{t_k}^{y_k}| = 0\}} \bigg) \mathbf{1}_{\big\{ \sum_{k=n}^{2n} \mathbf{1}_{\{|N_{t_k}^{y_k}| = 0\}} > 0\big\}}
\]
we have that
\begin{equation}
\label{cond_cs}
\bigg( E \Big( \sum_{k=n}^{2n} \mathbf{1}_{\{|N_{t_k}^{y_k}| = 0\}} \ \big\vert \mathcal{F}_{s_n} \Big) \bigg)^2 
\leq P \Big( \sum_{k=n}^{2n} \mathbf{1}_{\{|N_{t_k}^{y_k}| = 0\}} > 0 \ \big\vert \mathcal{F}_{s_n} \Big) 
E \bigg( \Big(\sum_{k=n}^{2n} \mathbf{1}_{\{|N_{t_k}^{y_k}| = 0\}} \Big)^2 \ \big\vert \mathcal{F}_{s_n} \bigg) \ P \text{-a.s.}
\end{equation}
Then noting that
\begin{align*}
E \bigg( \Big(\sum_{k=n}^{2n} \mathbf{1}_{\{|N_{t_k}^{y_k}| = 0\}} \Big)^2 \ \big\vert \mathcal{F}_{s_n} \bigg) 
\geq &P \Big( |N_{t_n}^{y_n}| = 0 \ \big\vert \mathcal{F}_{s_n} \Big)\\ 
= &\prod_{u \in N_{s_n}} P^{X^u_{s_n}} \Big( |N_{t_n - s_n}^{y_n}| = 0 \Big)\\ 
> &0 \quad P \text{-a.s.}
\end{align*}
we derive from \eqref{cond_cs} the conditional version of Paley-Zygmund inequality: 
\[
P \Big( \sum_{k=n}^{2n} \mathbf{1}_{\{|N_{t_k}^{y_k}| = 0\}} > 0 \ \big\vert \mathcal{F}_{s_n} \Big) \geq 
\frac{\bigg( E \Big( \sum_{k=n}^{2n} \mathbf{1}_{\{|N_{t_k}^{y_k}| = 0\}} \ \big\vert \mathcal{F}_{s_n} \Big) \bigg)^2}{E \bigg( \Big(\sum_{k=n}^{2n} \mathbf{1}_{\{|N_{t_k}^{y_k}| = 0\}} \Big)^2 \ \big\vert \mathcal{F}_{s_n} \bigg)} \qquad P \text{-a.s.} 
\] 
We may then substitute this in \eqref{main_liminf_upper2} to get
\begin{align}
\label{main_liminf_upper3}
&P \Big( \bigcup_{k = n}^{2n} \big\{ \big\vert N_{t_k}^{y_k} \big\vert = 0 \big\} \Big) \nonumber\\
\geq &E \Bigg[ \mathbf{1}_{\mathcal{A}_n} \frac{\Big( E \Big( \sum_{k=n}^{2n} \mathbf{1}_{\{|N_{t_k}^{y_k}| = 0\}} \ \big\vert \mathcal{F}_{s_n} \Big) \Big)^2}{E \Big( \Big(\sum_{k=n}^{2n} \mathbf{1}_{\{|N_{t_k}^{y_k}| = 0\}} \Big)^2 \ \big\vert \mathcal{F}_{s_n} \Big)} \Bigg] \nonumber\\
= &E \Bigg[ \mathbf{1}_{\mathcal{A}_n} \frac{\Big( \sum_{k=n}^{2n} P \Big( |N_{t_k}^{y_k}| = 0 \ \big\vert \mathcal{F}_{s_n} \Big) \Big)^2}{\sum_{k=n}^{2n} P \Big( |N_{t_k}^{y_k}| = 0 \ \big\vert \mathcal{F}_{s_n} \Big) + 2\sum_{n \leq k < l \leq 2n} P \Big( |N_{t_k}^{y_k}| = 0 , \ |N_{t_l}^{y_l}| = 0 \ \big\vert \mathcal{F}_{s_n} \Big)} \Bigg].
\end{align}
Let us suppose that the sequences of times $(s_n)_{n \geq 1}$ and $(t_n)_{n \geq 1}$ and the sequence of events $(\mathcal{A}_n)_{n \geq 1}$ can be chosen in such a way that
\begin{enumerate}
\item 
\begin{equation}
\label{condition_a}
P(\mathcal{A}_n) \to 1 \quad \text{ as } n \to \infty
\end{equation}
\item For all $n$, $k$ and $l$ such that $n \leq k < l \leq 2n$,
\begin{align}
\label{condition_b}
&P \Big( |N_{t_k}^{y_k}| = 0 , \ |N_{t_l}^{y_l}| = 0 \ \big\vert \mathcal{F}_{s_n} \Big) \nonumber\\
\leq &\gamma_n P \Big( |N_{t_k}^{y_k}| = 0 \ \big\vert \mathcal{F}_{s_n} \Big)
P \Big(|N_{t_l}^{y_l}| = 0 \ \big\vert \mathcal{F}_{s_n} \Big) \quad P (\cdot \ | \mathcal{A}_n) \text{-a.s.}
\end{align}
for some deterministic sequence $(\gamma_n)_{n \geq 1}$ such that $\gamma_n \to 1$ as $n \to \infty$.
\item For all $n \geq 1$,
\begin{equation}
\label{condition_c}
\sum_{k=n}^{2n} P \Big( |N_{t_k}^{y_k}| = 0 \ \big\vert \mathcal{F}_{s_n} \Big) \geq \theta_n \quad P (\cdot \ | \mathcal{A}_n) \text{-a.s.}
\end{equation}
for some deterministic sequence $(\theta_n)_{n \geq 1}$ such that $\theta_n \to \infty$ as $n \to \infty$.
\end{enumerate}
It would then follow from \eqref{main_liminf_upper3} that
\begin{align*}
&P \Big( \bigcup_{k = n}^{2n} \big\{ \big\vert N_{t_k}^{y_k} \big\vert = 0 \big\} \Big) \\
\geq &E \Bigg[ \mathbf{1}_{\mathcal{A}_n} \frac{\Big( \sum_{k=n}^{2n} P \Big( |N_{t_k}^{y_k}| = 0 \ \big\vert \mathcal{F}_{s_n} \Big) \Big)^2}{\sum_{k=n}^{2n} P \Big( |N_{t_k}^{y_k}| = 0 \ \big\vert \mathcal{F}_{s_n} \Big) + 2 \gamma_n \sum_{n \leq k < l \leq 2n} P \Big( |N_{t_k}^{y_k}| = 0 \ \big\vert \mathcal{F}_{s_n} \Big)
P \Big(|N_{t_l}^{y_l}| = 0 \ \big\vert \mathcal{F}_{s_n} \Big)} \Bigg]\\
\geq &E \Bigg[ \mathbf{1}_{\mathcal{A}_n} \frac{\Big( \sum_{k=n}^{2n} P \Big( |N_{t_k}^{y_k}| = 0 \ \big\vert \mathcal{F}_{s_n} \Big) \Big)^2}{\sum_{k=n}^{2n} P \Big( |N_{t_k}^{y_k}| = 0 \ \big\vert \mathcal{F}_{s_n} \Big) +  \gamma_n \Big( \sum_{k=n}^{2n} P \Big( |N_{t_k}^{y_k}| = 0 \ \big\vert \mathcal{F}_{s_n} \Big) \Big)^2} \Bigg]\\
\geq &P(\mathcal{A}_n) \frac{1}{\frac{1}{\theta_n} + \gamma_n}\\
\to &1
\end{align*}
as $n \to \infty$, which would establish the sought result.

Let us now take 
\[
t_n := n^3 \qquad , \ n \geq 1,
\]
\[
s_n := n \qquad , \ n \geq 1
\]
and
\[
\mathcal{A}_n := \big\{ (\log t_{2n})^{- \frac{\beta \epsilon}{2}} \leq M_{s_n} \leq (\log t_{2n})^{\frac{\beta \epsilon}{2}} , \ |X^u_{s_n}| \leq \beta s_n \text{ for all } u \in N_{s_n} \big\} \qquad , \ n \geq 1.
\]
Below we are going to check that \eqref{condition_a}, \eqref{condition_b} and \eqref{condition_c} are satisfied with this choice of $(t_n)_{n \geq 1}$, $(s_n)_{n \geq 1}$ and $(\mathcal{A}_n)_{n \geq 1}$.

\underline{Proof of \eqref{condition_a}:}

From \eqref{rightmost_as} and symmetry and also the fact that $M_{s_n}$ coverges almost surely to a strictly positive limit we have that 
\[
P(\mathcal{A}_n) \to 1 \quad \text{ as } n \to \infty.
\]

\underline{Proof of \eqref{condition_c}:}

Take any $n$ and $k$ such that $n \leq k \leq 2n$. Then by the Markov property, 
\[
P \Big( |N_{t_k}^{y_k}| = 0 \ \big\vert \mathcal{F}_{s_n} \Big) = \prod_{u \in N_{s_n}} P^{X^u_{s_n}} \Big( |N_{t_k - s_n}^{y_k}| = 0 \Big) \quad P \text{-a.s.}
\]
Also, from inequality \eqref{eq_expectation00},
\begin{align*}
P^{X^u_{s_n}} \Big( |N_{t_k - s_n}^{y_k}| > 0 \Big) \leq &E^{X^u_{s_n}} \big\vert N_{t_k - s_n}^{y_k} \big\vert\\
\leq &\mathrm{e}^{- \beta |X^u_{s_n}| - \beta y_k + \frac{\beta^2}{2}(t_k - s_n)} - \Phi \Big( - \frac{y_k - |X^u_{s_n}|}{\sqrt{t_k - s_n}}\Big)\\
= &\mathrm{e}^{- \beta |X^u_{s_n}| - \frac{\beta^2}{2} s_n} \big( \log t_k \big)^{1 - \beta \epsilon} - \Phi \Big( - \frac{y_k - |X^u_{s_n}|}{\sqrt{t_k - s_n}}\Big).
\end{align*}
Then on the event $\mathcal{A}_n$ we have that
\begin{align*}
&\mathrm{e}^{- \beta |X^u_{s_n}| - \frac{\beta^2}{2} s_n} \big( \log t_k \big)^{1 - \beta \epsilon} + \Phi \Big( - \frac{y_k - |X^u_{s_n}|}{\sqrt{t_k - s_n}}\Big)\\
= &\mathrm{e}^{- \beta |X^u_{s_n}| - \frac{\beta^2}{2} s_n} \big( \log t_k \big)^{1 - \beta \epsilon} \bigg( 1 + 
\mathrm{e}^{\beta |X^u_{s_n}| + \frac{\beta^2}{2} s_n} \big( \log t_k \big)^{- 1 + \beta \epsilon} \Phi \Big( - \frac{y_k - |X^u_{s_n}|}{\sqrt{t_k - s_n}}\Big) \bigg)\\
\leq &\mathrm{e}^{- \beta |X^u_{s_n}| - \frac{\beta^2}{2} s_n} \big( \log t_k \big)^{1 - \beta \epsilon} \bigg( 1 + 
\mathrm{e}^{\frac{3\beta^2}{2} s_n} \big( \log t_k \big)^{- 1 + \beta \epsilon} \Phi \Big( - \frac{y_k - \beta s_n}{\sqrt{t_k - s_n}}\Big) \bigg)\\
\leq &\theta_n' \mathrm{e}^{- \beta |X^u_{s_n}| - \frac{\beta^2}{2} s_n} \big( \log t_k \big)^{1 - \beta \epsilon},
\end{align*}
where
\[
\theta_n' := 1 + \mathrm{e}^{\frac{3\beta^2}{2} s_n} \big( \log t_n \big)^{- 1 + \beta \epsilon} \Phi \Big( - \frac{\frac{\beta}{2} t_n - (\frac{1}{\beta} - \epsilon)\log\log t_{2n}}{\sqrt{t_{2n} - s_n}}\Big) 
\to 1 \qquad \text{ as } n \to \infty.
\]
Thus, on $\mathcal{A}_n$, 
\begin{equation}
\label{ineq_000}
P^{X^u_{s_n}} \Big( |N_{t_k - s_n}^{y_k}| > 0 \Big) \leq \theta_n' \mathrm{e}^{- \beta |X^u_{s_n}| - \frac{\beta^2}{2} s_n} \big( \log t_k \big)^{1 - \beta \epsilon}
\end{equation}
and consequently
\[
P^{X^u_{s_n}} \Big( |N_{t_k - s_n}^{y_k}| = 0 \Big) \geq 1 - \theta_n' \mathrm{e}^{- \beta |X^u_{s_n}| - \frac{\beta^2}{2} s_n} \big( \log t_k \big)^{1 - \beta \epsilon}.
\]
Let us now note that for any $x^\ast \in (0, 1)$ and $x \in [0, x^\ast]$,
\[
1 - x \geq \exp \Big\{ \frac{\log(1 - x^\ast)}{x^\ast} x \Big\},
\]
which follows from the concavity of $\log (1-x)$ on the interval $[0, x^\ast]$. Furthermore,
\[
\frac{\log(1 - x^\ast)}{x^\ast} \to -1 \quad \text{ as } x^\ast \searrow 0.
\]
If we now take 
\[
x^\ast = \theta_n' \mathrm{e}^{- \frac{\beta^2}{2} s_n} \big( \log t_{2n} \big)^{1 - \beta \epsilon} 
\]
and 
\[
x = \theta_n' \mathrm{e}^{- \beta |X^u_{s_n}| - \frac{\beta^2}{2} s_n} \big( \log t_k \big)^{1 - \beta \epsilon} 
\]
then $P(\cdot \  | \mathcal{A}_n)$-a.s., for all $n$ sufficiently large that $\theta_n' \mathrm{e}^{- \frac{\beta^2}{2} s_n} \big( \log t_{2n} \big)^{1 - \beta \epsilon} < 1$, we have that 
\begin{align*}
P \Big( |N_{t_k}^{y_k}| = 0 \ \big\vert \mathcal{F}_{s_n} \Big)
\geq &\prod_{u \in N_{s_n}} \bigg( 1 - \theta_n' \mathrm{e}^{- \beta |X^u_{s_n}| - \frac{\beta^2}{2} s_n} \big( \log t_k \big)^{1 - \beta \epsilon} \bigg)\\
\geq &\prod_{u \in N_{s_n}} \exp \bigg\{ \frac{ \log \big( 1 - \theta_n' \mathrm{e}^{- \frac{\beta^2}{2} s_n} \big( \log t_{2n} \big)^{1 - \beta \epsilon} \big)}{\theta_n' \mathrm{e}^{- \frac{\beta^2}{2} s_n} \big( \log t_{2n} \big)^{1 - \beta \epsilon} } \theta_n' \mathrm{e}^{- \beta |X^u_{s_n}| - \frac{\beta^2}{2} s_n} \big( \log t_k \big)^{1 - \beta \epsilon} \bigg\}\\
\geq &\exp \big\{ \theta_n'' M_{s_n} \big( \log t_k \big)^{1 - \beta \epsilon} \big\},
\end{align*}
where
\[
\theta_n'' := \frac{ \log \big( 1 - \theta_n' \mathrm{e}^{- \frac{\beta^2}{2} s_n} \big( \log t_{2n} \big)^{1 - \beta \epsilon} \big)}{\theta_n' \mathrm{e}^{- \frac{\beta^2}{2} s_n} \big( \log t_{2n} \big)^{1 - \beta \epsilon} } \theta_n' \to {-1} \qquad \text{ as } n \to \infty.
\]
Then since $t_k \leq t_{2n}$ and since $M_{s_n} \leq (\log t_{2n})^{\frac{\beta \epsilon}{2}}$ on $\mathcal{A}_n$,  it follows that $P(\cdot \  
| \mathcal{A}_n)$-a.s., for all $n$ large enough, 
\[
\sum_{k = n}^{2n} P \Big( |N_{t_k}^{y_k}| = 0 \ \big\vert \mathcal{F}_{s_n} \Big) \geq n \exp \big\{ \theta_n'' \big( \log t_{2n} \big)^{1 - \frac{\beta \epsilon}{2}} \big\} \quad =: \theta_n,
\]
where
\[
\log \theta_n = \log n + \theta_n'' + \big( \log t_{2n} \big)^{1 - \frac{\beta \epsilon}{2}} \to \infty \qquad \text{ as } n \to \infty.
\]
This establishes \eqref{condition_c}.

\underline{Proof of \eqref{condition_b}:}

Take $n$, $k$ and $l$ such that $n \leq k < l \leq 2n$. Then from the Markov property we have
\begin{equation}
\label{eq_xxx}
P \Big( |N_{t_k}^{y_k}| = 0 , \ |N_{t_l}^{y_l}| = 0 \ \big\vert \mathcal{F}_{s_n} \Big) = 
\prod_{u \in N_{s_n}} P^{X^u_{s_n}} \Big( |N_{t_k - s_n}^{y_k}| = 0 , \ |N_{t_l - s_n}^{y_l}| = 0 \Big).
\end{equation}
We then rewrite each member of the above product as
\begin{align}
\label{eq_aaa}
&P^{X^u_{s_n}} \Big( |N_{t_k - s_n}^{y_k}| = 0 , \ |N_{t_l - s_n}^{y_l}| = 0 \Big)\nonumber\\
= &P^{X^u_{s_n}} \Big( |N_{t_k - s_n}^{y_k}| = 0 \Big) + P^{X^u_{s_n}} \Big( |N_{t_l - s_n}^{y_l}| = 0 \Big) - P^{X^u_{s_n}} \Big( \big\{ |N_{t_k - s_n}^{y_k}| = 0 \big\} \bigcup \big\{ |N_{t_l - s_n}^{y_l}| = 0 \big\} \Big)\nonumber\\
= &1 - P^{X^u_{s_n}} \Big( |N_{t_k - s_n}^{y_k}| > 0 \Big) - P^{X^u_{s_n}} \Big( |N_{t_l - s_n}^{y_l}| > 0 \Big) + P^{X^u_{s_n}} \Big( |N_{t_k - s_n}^{y_k}| |N_{t_l - s_n}^{y_l}| > 0 \Big). 
\end{align}
Proposition \ref{upper_bound_prop}, which we prove below, then says that there exists some positive sequence $(\eta_n)_{n \geq 1}$ converging to a finite limit such that, on the event $\mathcal{A}_n$,
\begin{equation}
\label{eq_bbb}
P^{X^u_{s_n}} \Big( |N_{t_k - s_n}^{y_k}| |N_{t_l - s_n}^{y_l}| > 0 \Big)
\leq \eta_n \mathrm{e}^{\beta |X^u_{s_n}|} P^{X^u_{s_n}} \Big( |N_{t_k - s_n}^{y_k}| > 0 \Big) P^{X^u_{s_n}} \Big( |N_{t_l - s_n}^{y_l}| > 0 \Big). 
\end{equation}
We also recall from \eqref{ineq_000} that for $n \leq k \leq 2n$, on the event $\mathcal{A}_n$,
\begin{align*}
P^{X^u_{s_n}} \Big( |N_{t_k - s_n}^{y_k}| > 0 \Big) \leq &\theta_n' \mathrm{e}^{ - \beta |X^u_{s_n}| - \frac{\beta^2}{2} s_n} \big( \log t_{2n} \big)^{1 - \beta \epsilon}\\
\leq &\frac{1}{2}
\end{align*}
for all $n$ sufficiently large and so 
\begin{equation}
\label{eq_ccc}
P^{X^u_{s_n}} \Big( |N_{t_k - s_n}^{y_k}| > 0 \Big) \leq 2 \theta_n' \mathrm{e}^{ - \beta |X^u_{s_n}| - \frac{\beta^2}{2} s_n} \big( \log t_{2n} \big)^{1 - \beta \epsilon} \Big( 1 -  P^{X^u_{s_n}} \Big( |N_{t_k - s_n}^{y_k}| > 0 \Big) \Big)
\end{equation}
for all $n$ sufficiently large. Combining \eqref{eq_aaa}, \eqref{eq_bbb} and \eqref{eq_ccc} we get that on $\mathcal{A}_n$, 
\begin{align}
\label{eq_yyy}
&P^{X^u_{s_n}} \Big( |N_{t_k - s_n}^{y_k}| = 0 , \ |N_{t_l - s_n}^{y_l}| = 0 \Big)\nonumber\\ 
\leq &1 - P^{X^u_{s_n}} \Big( |N_{t_k - s_n}^{y_k}| > 0 \Big) - P^{X^u_{s_n}} \Big( |N_{t_l - s_n}^{y_l}| > 0 \Big)\nonumber\\
&\qquad\qquad\qquad\qquad+ \eta_n \mathrm{e}^{\beta |X^u_{s_n}|} P^{X^u_{s_n}} \Big( |N_{t_k - s_n}^{y_k}| > 0 \Big) P^{X^u_{s_n}} \Big( |N_{t_l - s_n}^{y_l}| > 0 \Big)\nonumber\\
\leq &\bigg( 1 - P^{X^u_{s_n}} \Big( |N_{t_k - s_n}^{y_k}| > 0 \Big) \bigg) \bigg( 1 - P^{X^u_{s_n}} 
\Big( |N_{t_l - s_n}^{y_l}| > 0 \Big) \bigg)\nonumber\\ 
&\qquad\qquad\qquad\qquad + \eta_n \mathrm{e}^{\beta |X^u_{s_n}|} 
P^{X^u_{s_n}} \Big( |N_{t_k - s_n}^{y_k}| > 0 \Big) P^{X^u_{s_n}} \Big( |N_{t_l - s_n}^{y_l}| > 0 \Big)\nonumber\\
\leq &\big( 1 + \theta_n''' \mathrm{e}^{- \beta |X^u_{s_n}| - \beta^2 s_n} \big( \log t_{2n} \big)^{2 - 2 \beta \epsilon} \big)\bigg( 1 - P^{X^u_{s_n}} \Big( |N_{t_k - s_n}^{y_k}| > 0 \Big) \bigg) \bigg( 1 - P^{X^u_{s_n}} \Big( |N_{t_l - s_n}^{y_l}| > 0 \Big) \bigg)
\end{align}
for all $n$ sufficiently large and where $\theta_n''' = (2 \theta_n')^2 \eta_n$, $n \geq 1$ is a converging sequence. Therefore, substituting \eqref{eq_yyy} into \eqref{eq_xxx} and applying the trivial inequality $1+x \leq \mathrm{e}^x$, $x \in \mathbb{R}$, we get
\begin{align*}
&P \Big( |N_{t_k}^{y_k}| = 0 , \ |N_{t_l}^{y_l}| = 0 \ \big\vert \mathcal{F}_{s_n} \Big)\\ 
\leq &\prod_{u \in N_{s_n}} \bigg[ P^{X^u_{s_n}} \Big( |N_{t_k - s_n}^{y_k}| = 0 \Big) P^{X^u_{s_n}} 
\Big( |N_{t_l - s_n}^{y_l}| = 0 \Big) \bigg( 1 + \theta_n''' \mathrm{e}^{- \beta |X^u_{s_n}| - \beta^2 s_n} \big( \log t_{2n} \big)^{2 - 2\beta \epsilon}
\bigg) \bigg]\\
= &P \Big( |N_{t_k}^{y_k}| = 0 \ \big\vert \mathcal{F}_{s_n} \Big) P \Big( |N_{t_l}^{y_l}| = 0 \ \big\vert \mathcal{F}_{s_n} \Big) \prod_{u \in N_{s_n}} \bigg( 1 + \theta_n''' \mathrm{e}^{- \beta |X^u_{s_n}| - \beta^2 s_n} \big( \log t_{2n} \big)^{2 - 2\beta \epsilon}
\bigg)\\
\leq &P \Big( |N_{t_k}^{y_k}| = 0 \ \big\vert \mathcal{F}_{s_n} \Big) P \Big( |N_{t_l}^{y_l}| = 0 \ \big\vert \mathcal{F}_{s_n} \Big) \mathrm{e}^{\theta_n''' M_{s_n} \mathrm{e}^{- \frac{\beta^2}{2}s_n} ( \log t_{2n})^{2 - 2\beta \epsilon}}\\
\leq &\gamma_n P \Big( |N_{t_k}^{y_k}| = 0 \ \big\vert \mathcal{F}_{s_n} \Big) P \Big( |N_{t_l}^{y_l}| = 0 \ \big\vert \mathcal{F}_{s_n} \Big) \qquad P(\cdot \  | \mathcal{A}_n) \text{-a.s.}
\end{align*}
for all $n$ sufficiently large and where $\gamma_n = \exp \big\{ \theta_n''' \mathrm{e}^{- \frac{\beta^2}{2}s_n} ( \log t_{2n})^{2 - \frac{5\beta \epsilon}{2}} \big\} \to 1$ as $n \to \infty$. This verifies \eqref{condition_b} and finishes the proof of Proposition \ref{main_liminf_upper}.
\end{proof}
It remains to establish \eqref{eq_bbb}, which is dealt with by the following proposition.
\begin{Proposition}
\label{upper_bound_prop}
Let $t_n =  n^3$, $s_n = n$ and $y_n = \frac{\beta}{2} t_n - \big( \frac{1}{\beta} - \epsilon \big) \log \log t_n$, $n \geq 1$ as above. Then there exists a positive converging sequence $(\eta_n)_{n \geq 1}$ such that for all $k$, $l$ and $x_0$ with $n \leq k < l \leq 2n$ and $|x_0| \leq \beta s_n$, 
\begin{equation}
\label{upper_bound_prop_eq}
P^{x_0} \Big( |N_{t_k - s_n}^{y_k}| |N_{t_l - s_n}^{y_l}| > 0 \Big) \leq \eta_n \mathrm{e}^{\beta |x_0|} P^{x_0} \Big( |N_{t_k - s_n}^{y_k}| > 0 \Big) 
P^{x_0} \Big( |N_{t_l - s_n}^{y_l}| > 0 \Big).
\end{equation}
\end{Proposition}
\begin{proof}
From Markov and Paley-Zygmund inequalities we have
\begin{align*}
&P^{x_0} \Big( |N_{t_k - s_n}^{y_k}| |N_{t_l - s_n}^{y_l}| > 0 \Big)\\
\leq &E^{x_0} \Big[ |N_{t_k - s_n}^{y_k}| |N_{t_l - s_n}^{y_l}| \Big]\\
\leq &E^{x_0} \Big[ |N_{t_k - s_n}^{y_k}| |N_{t_l - s_n}^{y_l}| \Big] P^{x_0} \Big( |N_{t_k - s_n}^{y_k}| > 0 \Big) 
P^{x_0} \Big( |N_{t_l - s_n}^{y_l}| > 0 \Big) \frac{E^{x_0} \big[ |N_{t_k - s_n}^{y_k}|^2 \big]}{\Big( E^{x_0} |N_{t_k - s_n}^{y_k}| \Big)^2} \frac{E^{x_0} \big[ |N_{t_l - s_n}^{y_k}|^2 \big]}{\Big( E^{x_0} |N_{t_l - s_n}^{y_k}| \Big)^2}.
\end{align*}
Below we are going to show that for all $k$, $l$ and $x_0$ with $n \leq k < l \leq 2n$ and $|x_0| \leq \beta s_n$,
\begin{equation}
\label{upper_bound_prop_eq1}
E^{x_0} \big[ |N_{t_k - s_n}^{y_k}|^2 \big] \leq \eta_n' E^{x_0} |N_{t_k - s_n}^{y_k}|
\end{equation}
and
\begin{equation}
\label{upper_bound_prop_eq2}
E^{x_0} \Big[ |N_{t_k - s_n}^{y_k}| |N_{t_l - s_n}^{y_l}| \Big] \leq \eta_n'' \mathrm{e}^{\beta |x_0|} E^{x_0} |N_{t_k - s_n}^{y_k}| E^{x_0} |N_{t_l - s_n}^{y_l}|
\end{equation}
for some positive converging sequences $(\eta_n')_{n \geq 1}$ and $(\eta_n'')_{n \geq 1}$. This will yield \eqref{upper_bound_prop_eq} with $\eta_n = \eta_n'' (\eta_n')^2$.

\underline{Proof of \eqref{upper_bound_prop_eq1}}:

From \eqref{eq_expectation}, 
\begin{align}
\label{upper_bound_prop_eq3}
E^{x_0} |N_{t_k - s_n}^{y_k}| \geq &\mathrm{e}^{- \beta |x_0| - \beta y_k + \frac{\beta^2}{2}(t_k - s_n)}
\Phi \Big( \frac{\beta(t_k - s_n) - |x_0| - y_k}{\sqrt{t_k - s_n}} \Big)\nonumber\\
= &\mathrm{e}^{- \beta |x_0| - \frac{\beta^2}{2} s_n} \big( \log t_k \big)^{1 - \beta \epsilon}
\Phi \Big( \frac{ \frac{\beta}{2}t_k - \beta s_n - |x_0| + (\frac{1}{\beta} - \epsilon)\log \log t_k}{\sqrt{t_k - s_n}} \Big)\nonumber\\
\geq &\eta_n''' \mathrm{e}^{- \beta |x_0| - \frac{\beta^2}{2} s_n} \big( \log t_k \big)^{1 - \beta \epsilon},
\end{align}
where
\[
\eta_n''' = \Phi \Big( \frac{ \frac{\beta}{2}t_n - 2 \beta s_n + (\frac{1}{\beta} - \epsilon)\log \log t_n}{\sqrt{t_{2n} - s_n}} \Big) \to 1 \quad \text{ as } n \to \infty.
\]
Also, from Proposition \ref{second_moment_prop} and then \eqref{upper_bound_prop_eq3}, 
\begin{align*}
E^{x_0} \big[ |N_{t_k - s_n}^{y_k}|^2 \big] \leq &E^{x_0} |N_{t_k - s_n}^{y_k}| + 8 \mathrm{e}^{- \beta |x_0| - 2 \beta y_k + \beta^2 (t_k - s_n)}\\
= &E^{x_0} |N_{t_k - s_n}^{y_k}| + 8 \mathrm{e}^{- \beta |x_0| - \beta^2 s_n} (\log t_k)^{2 - 2\beta \epsilon}\\
\leq &E^{x_0} |N_{t_k - s_n}^{y_k}| \Big( 1 + \frac{8}{\eta_n'''} \mathrm{e}^{- \frac{\beta^2}{2}s_n} (\log t_k)^{1 - \beta \epsilon}\Big)\\
\leq &\eta_n' E^{x_0} |N_{t_k - s_n}^{y_k}|, 
\end{align*}
where
\[
\eta_n' = 1 + \frac{8}{\eta_n'''} \mathrm{e}^{- \frac{\beta^2}{2}s_n} (\log t_{2n})^{1 - \beta \epsilon} \to 1 \quad \text{ as } n \to \infty.
\]

\underline{Proof of \eqref{upper_bound_prop_eq2}}:

From Proposition \ref{second_moment_prop3} and inequality \eqref{upper_bound_prop_eq3} we have that 
\begin{align*}
E^{x_0} \Big[ \big\vert N_{t_k - s_n}^{y_k} \big\vert \big\vert N_{t_l - s_n}^{y_l} \big\vert \Big] \leq &24 \mathrm{e}^{- \beta |x_0| - \beta y_k - \beta y_l + \frac{\beta^2}{2}(t_k-s_n) + \frac{\beta^2}{2}(t_l-s_n)}\\ 
&\qquad\qquad\qquad + \sum_{i = 1}^4 e_i(x_0, t_k - s_n, t_l - s_n, y_k, y_l)\\
\leq &\frac{24}{(\eta_n''')^2} \mathrm{e}^{\beta |x_0|} E^{x_0} \big\vert N_{t_k-s_n}^{y_k} \big\vert E^{x_0} \big\vert N_{t_l-s_n}^{y_l} \big\vert\\ 
&\qquad\qquad\qquad + \sum_{i = 1}^4 e_i(x_0, t_k - s_n, t_l - s_n, y_k, y_l).
\end{align*}
Below we are going to check that the contribution from the $e_i$ terms is negligible.

Firstly, 
\begin{align*}
&\mathrm{e}^{\frac{\beta^2}{2}s_n} (\log t_l)^{-1 + \beta \epsilon} \Phi \Big( - \frac{y_l - y_k}{\sqrt{t_l - t_k}} \Big)\\
 = &\mathrm{e}^{\frac{\beta^2}{2}s_n} (\log t_l)^{-1 + \beta \epsilon} \Phi \Big( - \frac{\beta}{2} \sqrt{t_l - t_k} + \big( \frac{1}{\beta} - \epsilon\big)  \frac{\log \log t_l - \log \log t_k}{\sqrt{t_l - t_k}} \Big)\\
\leq &\mathrm{e}^{\frac{\beta^2}{2}s_n} (\log t_{2n})^{-1 + \beta \epsilon} \Phi \Big( - \frac{\beta}{2} \sqrt{t_{n+1} - t_n} + \big( \frac{1}{\beta} - \epsilon\big)  \log \log t_{2n} \Big)\\
\leq &\mathrm{e}^{\frac{\beta^2}{2}n} \Big( \log \big((2n)^3\big) \Big)^{-1 + \beta \epsilon} \Phi \Big( - \frac{\beta}{2} n  + \big( \frac{1}{\beta} - \epsilon\big)  \log \log \big((2n)^3\big) \Big) \quad =: \eta_n^{(iv)} \quad (\to 0 \text{ as } n \to \infty)
\end{align*}
and hence
\begin{align*}
&e_1(x_0, t_k-s_n, t_l-s_n, y_k, y_l)\\ 
= &\bigg( \mathrm{e}^{- \beta y_l + \frac{\beta^2}{2}(t_l - t_k)} + \Phi \Big( - \frac{y_l-y_k}{\sqrt{t_l-t_k}} \Big) \bigg) E^{x_0} \big\vert N_{t_k - s_n}^{y_k} \big\vert\\
= &\mathrm{e}^{\beta |x_0|} \bigg( \mathrm{e}^{- \frac{\beta^2}{2}(t_k-s_n)} + \mathrm{e}^{\frac{\beta^2}{2}s_n} (\log t_l)^{-1 + \beta \epsilon} \Phi \Big( - \frac{y_l-y_k}{\sqrt{t_l-t_k}} \Big) \bigg) \mathrm{e}^{- \beta|x_0| - \frac{\beta^2}{2}s_n } \big( \log t_l \big)^{1 - \beta \epsilon}  E^{x_0} \big\vert N_{t_k - s_n}^{y_k} \big\vert\\
\leq &\mathrm{e}^{\beta |x_0|} \Big( \mathrm{e}^{- \frac{\beta^2}{2}(t_{2n}-s_n)} + \eta_n^{(iv)} \Big) \frac{1}{\eta_n'''} E^{x_0} \big\vert N_{t_l - s_n}^{y_l} \big\vert  E^{x_0} \big\vert N_{t_k - s_n}^{y_k} \big\vert\\
= &\eta_n^{(v)} \mathrm{e}^{\beta |x_0|} E^{x_0} \big\vert N_{t_k - s_n}^{y_k} \big\vert  E^{x_0} \big\vert N_{t_l - s_n}^{y_l} \big\vert,
\end{align*}
where 
\[
\eta_n^{(v)} = \frac{1}{\eta_n'''} \Big( \mathrm{e}^{- \frac{\beta^2}{2}(t_{2n}-s_n)} + \eta_n^{(iv)} \Big) \to 0 \quad \text{ as } n \to \infty.
\]
Also, 
\begin{align*}
\mathrm{e}^{\beta y_l - \frac{\beta^2}{2}(t_l - t_k)} \Phi \Big( - \frac{y_l}{\sqrt{t_l - t_k}} \Big) \leq  
&\mathrm{e}^{\beta y_l - \frac{\beta^2}{2}(t_l - t_k)} \frac{1}{\sqrt{2 \pi}} \frac{\sqrt{t_l - t_k}}{y_l} \mathrm{e}^{- \frac{1}{2} \frac{y_l^2}{t_l - t_k}}\\
= &\frac{1}{\sqrt{2 \pi}} \frac{\sqrt{t_l - t_k}}{y_l} \mathrm{e}^{- \frac{1}{2(t_l - t_k)} \big(y_l - \beta(t_l - t_k)\big)^2}\\
\leq &\frac{1}{\sqrt{2 \pi}} \frac{\sqrt{t_{2n}}}{y_n} \ =: \eta_n^{(vi)} \quad (\to 0 \text{ as } n \to \infty)
\end{align*}
and hence 
\begin{align*}
e_2(x_0, t_k - s_n, t_l - s_n, y_k, y_l) = &16 \mathrm{e}^{- \beta |x_0| - \beta y_k + \beta^2 (t_k-s_n)} \Phi \Big( - \frac{y_l}{\sqrt{t_l - t_k}} \Big)\\
= &16 \mathrm{e}^{\beta |x_0|} \mathrm{e}^{- \beta |x_0| - \beta y_k + \frac{\beta^2}{2}(t_k - s_n)} \mathrm{e}^{- \beta |x_0| - \beta y_l + \frac{\beta^2}{2}(t_l - s_n)}\\
&\qquad\qquad\qquad \times \mathrm{e}^{\beta y_l - \frac{\beta^2}{2}(t_l - t_k)} \Phi \Big( - \frac{y_l}{\sqrt{t_l - t_k}} \Big)\\
\leq &\frac{16 \eta_n^{(vi)}}{(\eta_n''')^2} \mathrm{e}^{\beta |x_0|} E^{x_0} \big\vert N_{t_k - s_n}^{y_k} \big\vert E^{x_0} \big\vert N_{t_l - s_n}^{y_l} \big\vert.
\end{align*}
Also,
\begin{align*}
&\mathrm{e}^{\beta y_k - \frac{\beta^2}{2}(t_k-s_n)} \Phi \Big( \frac{y_l - y_k}{\sqrt{t_l - t_k}} - \beta \sqrt{t_l - t_k} \Big)\\
= &\mathrm{e}^{\frac{\beta^2}{2}s_n} (\log t_k)^{-1 + \beta \epsilon} \Phi \Big( - \frac{\beta}{2} \sqrt{t_l - t_k} - \big( \frac{1}{\beta} - \epsilon \big) \frac{\log \log t_l - \log \log t_k}{\sqrt{t_l - t_k}} \Big)\\
\leq &\mathrm{e}^{\frac{\beta^2}{2}s_n} (\log t_n)^{-1 + \beta \epsilon} \Phi \Big( - \frac{\beta}{2} \sqrt{t_{n+1} - t_n} \Big)\\
\leq &\mathrm{e}^{\frac{\beta^2}{2}n} \big( \log (n^3) \big)^{-1 + \beta \epsilon} \Phi \big( - \frac{\beta}{2} n \big) \ =: \eta_n^{(vii)} \ (\to 0 \text{ as } n \to \infty)
\end{align*}
and hence 
\begin{align*}
e_3(x_0, t_k - s_n, t_l - s_n, y_k, y_l) = &\mathrm{e}^{- \beta |x_0 | - \beta y_l + \frac{\beta^2}{2}(t_l - s_n)} \Phi \Big( \frac{y_l - y_k}{\sqrt{t_l - t_k}} - \beta \sqrt{t_l - t_k} \Big)\\ 
= &\mathrm{e}^{\beta |x_0|} \mathrm{e}^{- \beta |x_0 | - \beta y_l + \frac{\beta^2}{2}(t_l - s_n)} \mathrm{e}^{- \beta |x_0 | - \beta y_k + \frac{\beta^2}{2}(t_k - s_n)}\\
&\qquad\qquad\qquad \times \mathrm{e}^{\beta y_k - \frac{\beta^2}{2}(t_k-s_n)} \Phi \Big( \frac{y_l - y_k}{\sqrt{t_l - t_k}} - \beta \sqrt{t_l - t_k} \Big)\\
\leq &\frac{\eta_n^{(vii)}}{(\eta_n''')^2} \mathrm{e}^{\beta |x_0|} E^{x_0} \big\vert N_{t_k - s_n}^{y_k} \big\vert E^{x_0} \big\vert N_{t_l - s_n}^{y_l} \big\vert.
\end{align*}
Finally,
\begin{align*}
&\mathrm{e}^{\beta |x_0| + \beta y_k + \beta y_l - \frac{\beta^2}{2}(t_k-s_n) - \frac{\beta^2}{2}(t_l-s_n)} \Phi \Big( - \frac{y_k - x_0}{\sqrt{t_k - s_n}} \Big)\\ 
= &\mathrm{e}^{\beta |x_0| + \beta^2 s_n} (\log t_k)^{- 1 + \beta \epsilon} (\log t_l)^{- 1 + \beta \epsilon} \Phi \Big( - \frac{\beta}{2} \frac{t_k}{\sqrt{t_k - s_n}}+ \frac{( \frac{1}{\beta} - \epsilon ) \log \log t_k + x_0}{\sqrt{t_k - s_n}}\Big)\\ 
\leq &\mathrm{e}^{2 \beta^2 s_n}(\log t_n)^{-2 + 2\beta \epsilon}\Phi \Big( - \frac{\beta}{2} \sqrt{t_n} + \frac{(\frac{1}{\beta} - \epsilon)\log \log t_{2n} + \beta s_n}{\sqrt{t_n - s_n}}\Big)\\ 
= &\mathrm{e}^{2 \beta^2 n} \big( \log (n^3) \big)^{-2 + 2 \beta \epsilon} \Phi \Big( - \frac{\beta}{2} n^{\frac{3}{2}} + \frac{(\frac{1}{\beta} - \epsilon)\log \log \big( (2n)^3 \big) + \beta n}{\sqrt{n^3 - n}} \Big) \quad =: \eta_n^{(viii)} \ (\to 0 \text{ as } n \to \infty)
\end{align*}
and hence
\begin{align*}
e_4(x_0, t_k - s_n, t_l - s_n, y_k, y_l) = &\mathrm{e}^{\beta |x_0|} \mathrm{e}^{- \beta |x_0 | - \beta y_l + \frac{\beta^2}{2}(t_l - s_n)} \mathrm{e}^{- \beta |x_0 | - \beta y_k + \frac{\beta^2}{2}(t_k - s_n)}\\
&\qquad\qquad \times \mathrm{e}^{\beta |x_0| + \beta y_k + \beta y_l - \frac{\beta^2}{2}(t_k-s_n) - \frac{\beta^2}{2}(t_l-s_n)} \Phi \Big( - \frac{y_k - x_0}{\sqrt{t_k - s_n}} \Big)\\
\leq &\frac{\eta_n^{(viii)}}{(\eta_n''')^2} \mathrm{e}^{\beta |x_0|} E^{x_0} \big\vert N_{t_k - s_n}^{y_k} \big\vert E^{x_0} \big\vert N_{t_l - s_n}^{y_l} \big\vert.
\end{align*}

\end{proof}

\end{document}